\documentclass{article}

\PassOptionsToPackage{numbers, compress}{natbib}

\usepackage[preprint]{neurips_2026}

\usepackage[utf8]{inputenc} 
\usepackage[T1]{fontenc}    
\usepackage{hyperref}       
\usepackage{url}            
\usepackage{booktabs}       
\usepackage{amsfonts}       
\usepackage{nicefrac}       
\usepackage{xcolor}         
\usepackage{graphicx}
\usepackage{setspace}
\usepackage{algorithm}
\usepackage{algorithmic}
\usepackage{caption}

\usepackage{amsmath}
\usepackage{amssymb}
\usepackage{mathtools}
\usepackage{amsthm}
\usepackage{times,amsmath,amsbsy,amssymb,amscd,mathrsfs}
\usepackage{multicol,multirow}
\usepackage{mathtools}
\usepackage[dvipsnames,svgnames,table]{xcolor}
\usepackage[all]{xy}
\usepackage{wrapfig}
\usepackage{tcolorbox}
\usepackage{stmaryrd}

\usepackage{tikz,tikz-cd}
\usepackage[utf8]{inputenc}
\usepackage{pgfplots} 
\usepackage{pgfgantt}
\usepackage{pdflscape}
\pgfplotsset{compat=newest} 
\pgfplotsset{plot coordinates/math parser=false}
\newlength\fwidth

\usepackage{listings}
\usepackage{xcolor}
\definecolor{myBlue}{rgb}{0.0,0.0,0.55}

\usepackage[hyperpageref]{backref}
\usepackage{circledsteps}

\usepackage{comment,enumerate,multicol,xspace}

  \newcounter{mnote}
  \let\oldmarginpar\marginpar
    \renewcommand\marginpar[1]{\-\oldmarginpar[\raggedleft\footnotesize #1]%
    {\raggedright\footnotesize #1}}

\newcommand{\dd}{\,{\rm d}}

\newcommand{\dual}[1]{\left\langle {#1} \right\rangle}

\newcommand{\vertiii}[1]{{\left\vert\kern-0.25ex\left\vert\kern-0.25ex\left\vert #1 
    \right\vert\kern-0.25ex\right\vert\kern-0.25ex\right\vert}}

\usepackage[capitalize,noabbrev]{cleveref}

\theoremstyle{plain}
\newtheorem{theorem}{Theorem}[section]
\newtheorem{proposition}[theorem]{Proposition}
\newtheorem{lemma}[theorem]{Lemma}

\theoremstyle{definition}
\newtheorem{definition}[theorem]{Definition}
\newtheorem{assumption}[theorem]{Assumption}
\theoremstyle{remark}
\newtheorem{remark}[theorem]{Remark}
\newtheorem{example}[theorem]{Example}

\title{Adaptive Accelerated Mirror Descent in Primal and Dual Spaces}

\author{%
  Zeyi Xu \\
  Department of Mathematics\\
  University of California, Irvine\\
  Irvine, CA 92697 \\
  \texttt{zeyix1@uci.edu} \\
 \And
 Long Chen \\
 Department of Mathematics\\
  University of California, Irvine\\
  Irvine, CA 92697 \\
  \texttt{chenlong@math.uci.edu} \\
}

\begin{document}

\maketitle

\begin{abstract}
We propose Adaptive Accelerated Mirror Descent (AAMD), a flow-based method that combines nonlinear preconditioning, acceleration, and adaptivity in mirror geometry. The key ingredient is an accumulated Lyapunov perturbation budget: local descent failures are allowed as long as the total budget remains nonpositive, so line search is used only when stability is at risk. We prove accelerated convergence under dual relative smoothness/convexity and a mirror-geometry compatibility condition, and obtain an $O(1/k^2)$ rate for convex objectives by homotopy under a bounded-sublevel-set assumption. Experiments on relative-smoothness problems show that combining preconditioning, acceleration, and adaptivity gives substantial gains over methods using only part of this structure.
\end{abstract}

\section{Introduction}

Gradient descent (GD) is a basic tool in first-order optimization. Its performance can degrade sharply when $f$ is ill-conditioned or when the Euclidean geometry poorly matches the problem. Three ideas address this issue: preconditioning, acceleration, and adaptivity. Each is well developed on its own. Much less is known about how to combine all three in nonlinear mirror geometry.

\paragraph{Preconditioning and generalized smoothness.}
Standard preconditioning replaces the Euclidean metric by a state-dependent matrix $A_k$, giving
$x_{k+1}=x_k-\alpha_k A_k^{-1}\nabla f(x_k)$. Mirror descent (MD)~\citep{nemirovskij1983problem} generalizes this idea by using a mirror function $\phi$:
\begin{equation}\label{eq:MD}
\nabla \phi(x_{k+1})=\nabla \phi(x_k)-\alpha_k\nabla f(x_k).
\end{equation}
Dual mirror descent (DMD) uses the Fenchel conjugate $\phi^*$ and gives the primal-space update
\begin{equation}\label{eq:DMD}
x_{k+1}=x_k-\alpha_k\nabla \phi^*(\nabla f(x_k)).
\end{equation}
DMD includes gradient clipping and coordinate-wise schemes, such as Adam and Adagrad, by allowing larger effective steps along low-curvature directions~\citep{OikonoQuanLaudePatrin2025Nonlinearly}. In particular, \citet{maddison2021dual} established the formal duality between MD~\eqref{eq:MD} and DMD~\eqref{eq:DMD}.

\paragraph{Acceleration under relative geometry.}
Momentum acceleration, pioneered by Nesterov~\cite{nesterov1983method}, achieves the optimal $\mathcal O(1/k^2)$ rate for convex smooth problems. Classical acceleration is Euclidean. Its extension to mirror geometry has been studied in several forms~\cite{nemirovski2004prox,nesterov2005smooth}. \citet{kim2023mirror} studied mirror duality in convex optimization and derived dual accelerated mirror-descent-type methods. ODE discretizations also give accelerated rates~\cite{krichene2015accelerated,yuan2023analysis}. However, these rates are of order $\mathcal O(L_f/k^2)$ and still depend on the Euclidean Lipschitz constant $L_f$ of $\nabla f$. This constant can be large or infinite, which limits the benefit of mirror geometry.

To avoid this Euclidean dependence, \citet{Hanzely2021} proposed Accelerated Bregman Proximal Gradient (ABPG), whose rate depends on the triangle scaling exponent (TSE). More recently, \citet{chen2025accmd} developed Acc-MD through variable-operator splitting~\cite{chen2025accelerated} and proved an accelerated linear rate under relative geometry.

Acceleration in mirror geometry requires extra structure. Relative smoothness alone does not generally imply global accelerated rates~\cite{DragomTaylordAspreBolte2022Optimal}. Existing approaches impose additional geometric conditions, such as the TSE condition in ABPG or the generalized Cauchy--Schwarz condition in~\cite{chen2025accelerated}. 

\paragraph{Adaptivity without overhead.}
Practical methods must adapt to unknown or spatially varying curvature. Classical line search enforces descent at every step, but this may require repeated gradient evaluations. Modern adaptivity, as defined by~\cite{malitsky2024adaptive}, tracks local smoothness using only information already computed. Such schemes exist for Euclidean GD~\citep{malitsky2019adaptive,xu2026adaptiveacceleratedgradientdescent}, but they have not been fully developed for accelerated mirror methods.

\paragraph{Main idea.}
The three ingredients above are not new by themselves. The contribution of this paper is to combine them in one stable algorithmic framework. The method uses mirror geometry as a nonlinear preconditioner, adds acceleration through a primal--dual mirror flow, and adapts local parameters through a Lyapunov stability budget. This combination matters: preconditioning alone changes the geometry but does not accelerate; acceleration alone remains tied to the Euclidean smoothness constant $L_f$ and may fail to apply when $L_f$ is large or infinite; adaptivity alone does not exploit the mirror geometry. Their coupling gives a method that is geometry-aware, accelerated under relative geometry, and stable with few backtracking steps.

\paragraph{Contributions.}
The paper makes the following contributions.

\noindent\textbf{1. A primal--dual accelerated mirror framework.}
We derive an accelerated mirror-descent-type scheme by coupling a dual preconditioned descent step for the primal variable with a primal mirror update for an auxiliary variable. The construction is guided by a continuous-time Lyapunov flow. This flow shows how nonlinear preconditioning and acceleration can be coupled while preserving the energy dissipation property.

The starting point is closely related to dual-space preconditioning and mirror duality~\cite{maddison2021dual,laude2025anisotropic,kim2023mirror}. In particular, a related coupling of primal and dual mirror updates was considered by~\citet{kim2023mirror}. However, their method does not include adaptivity, and its acceleration rate depends on the Euclidean Lipschitz constant $L_f$ rather than a relative smoothness constant. The new point here is the flow-based integration of nonlinear preconditioning, acceleration, and adaptive stability control.

\noindent\textbf{2. Adaptive stability through an accumulated perturbation budget.}
The main algorithmic ingredient is an accumulated perturbation budget. Instead of imposing a strict descent inequality at every step, the method allows local violations when they are compensated by previous descent. Backtracking is triggered only when the accumulated budget becomes positive. Thus line search acts as a safeguard rather than a repeated per-iteration routine, while global Lyapunov stability is preserved. In our numerical experiments, fewer than ten backtracking steps occur in total, and all appear during the initial transient phase.

\noindent\textbf{3. A combined method with improved practical performance.}
The resulting Adaptive Accelerated Mirror Descent (AAMD) method combines nonlinear preconditioning, acceleration, and adaptivity. Numerical experiments show that this combination is stronger than using only one or two of these ingredients.

%
%
%

\paragraph{Limitations.}
The results require a mirror map adapted to the objective. The method may not help if $\nabla\phi$ or $\nabla\phi^*$ is expensive to compute, or if the mirror map does not reduce the relative condition number. The accelerated rate also relies on a mirror-geometry compatibility condition, which is automatic in the Euclidean case but may hold only locally or with a large constant for general mirror maps. The present analysis is deterministic and convex. Stochastic and nonconvex extensions are left for future work.

\section{Preliminaries}\label{sec: Preliminaries}

We recall basic facts from finite-dimensional convex optimization and duality; see, e.g., \cite{bertsekas2009convex,rockafellar1970convex}. Throughout the paper, $V$ denotes a finite-dimensional Hilbert space and $V^*$ denotes its dual space. The duality pairing between $\chi\in V^*$ and $x\in V$ is denoted by $\langle \chi,x\rangle$. By the Riesz representation theorem, $V$ can be identified with $V^*$, in which case the duality pairing coincides with the inner product $(\cdot,\cdot)$.

\paragraph{Bregman divergence.}
For a continuously differentiable function $f$, the {\it Bregman divergence} is
\begin{equation*}
D_f(y,x):=f(y)-f(x)-\langle \nabla f(x),y-x\rangle.
\end{equation*}
For $f\in \mathcal C^1(V)$, convexity is equivalent to $D_f(y,x)\ge 0$ for all $x,y\in V$. If $f$ is strictly convex, then $D_f(y,x)=0$ if and only if $x=y$. In general, $D_f$ is not symmetric. Its symmetrization satisfies
$$
D_f(y,x)+D_f(x,y)=\langle \nabla f(y)-\nabla f(x),y-x\rangle.
$$
A key tool in the analysis is the three-point identity~\cite{chen1993convergence}:
\begin{equation}\label{eq:Bregmanidentity}
\langle \nabla f(z)-\nabla f(y),x-y\rangle
=
D_f(x,y)+D_f(y,z)-D_f(x,z).
\end{equation}
For a Legendre function $\phi$, the Bregman divergences of $\phi$ and its Fenchel conjugate $\phi^*$ satisfy
\begin{equation}\label{eq:DfDf*}
D_\phi(x,y)=D_{\phi^*}(\nabla\phi(y),\nabla\phi(x)),
\end{equation}
with reversed arguments. The gradient of the Bregman divergence with respect to the first variable is
\begin{equation}\label{eq:gradD}
\begin{aligned}
&\nabla D_f(\cdot,x)=\nabla f(\cdot)-\nabla f(x),\\
&\nabla D_{\phi^*}(\cdot,\chi)=\nabla\phi^*(\cdot)-\nabla\phi^*(\chi).
\end{aligned}
\end{equation}

\paragraph{Mirror duality.}
Throughout the paper, we assume $\nabla\phi^*(0)=0$ and $\nabla\phi(0)=0$, which can be achieved by a shift. The mirror function $\phi$ is assumed to be Legendre: it is closed, proper, strictly convex, differentiable on $\operatorname{int}\operatorname{dom}\phi$, and its gradient map is one-to-one from $\operatorname{int}\operatorname{dom}\phi$ onto $\operatorname{int}\operatorname{dom}\phi^*$. Hence
$$
(\nabla\phi)^{-1}=\nabla\phi^*.
$$
When we use $(\nabla f)^{-1}=\nabla f^*$ in the dual formulation, we assume the corresponding Legendre-type condition for $f$ on the relevant domain.

For the primal problem $\min_{x\in V} f(x)$, consider the dual mirror descent step for $\min_{\chi\in V^*}\phi^*(\chi)$ with $f^*$ as the mirror function:
\begin{equation}\label{eq:dualmd}
\nabla f^*(\chi_{k+1})=\nabla f^*(\chi_k)-\alpha_k\nabla\phi^*(\chi_k).
\end{equation}
This scheme is formally dual to the primal mirror descent update~\eqref{eq:MD}, and we call it dual mirror descent (DMD). Identifying $x_k=\nabla f^*(\chi_k)$, or equivalently $\chi_k=\nabla f(x_k)$, transforms~\eqref{eq:dualmd} into the primal update~\eqref{eq:DMD}. 
Thus, parts of the mirror descent analysis can be transferred to DMD. 

\paragraph{Relative smoothness in the dual space.}
Relative smoothness and relative convexity, first introduced in~\cite{lu2018relatively}, are standard assumptions for mirror descent methods.

\noindent\textbf{Assumption (A1).} There exist constants $0\le \mu\le L$ such that
\begin{equation}\label{eq:relativesmoothness}
\mu D_\phi(x,y)\le D_f(x,y)\le L D_\phi(x,y).
\end{equation}
Assumption (A1) is equivalent to the convexity of $f-\mu\phi$ and $L\phi-f$.

In parallel, DMD requires the following dual relative smoothness and convexity condition~\cite{maddison2021dual}, also called anisotropic smoothness in~\cite{laude2025anisotropic}.

\noindent\textbf{Assumption (A1$^*$).} There exist constants $0\le \mu\le L$ such that
\begin{equation}\label{eq:Lmu-star-primal}
\mu D_f(y,x)\le D_{\phi^*}(\nabla f(x),\nabla f(y))\le L D_f(y,x).
\end{equation}
Assumption (A1$^*$) is the standard relative smoothness and convexity condition for the dual pair $(\phi^*,f^*)$: $\phi^*-\mu f^*$ and $Lf^*-\phi^*$ are convex. Its equivalence to anisotropic smoothness was proved in~\cite{laude2025anisotropic} using Moreau envelope arguments.

The constants in (A1) and (A1$^*$) need not be the same. Indeed, the two assumptions are independent, and we specify which one is used in each result. When $\phi(x)=\frac12\|x\|^2$, both reduce to the usual Euclidean smoothness and strong convexity assumptions. For nonquadratic $\phi$, Assumptions  (A1) and (A1$^*$) can capture non-Euclidean curvature even when $\nabla f$ is not Lipschitz. The assumption is useful but restrictive: an adapted mirror map must be available, and both $\nabla\phi$ and $\nabla\phi^*$ must be efficient to compute.

\section{Accelerated Primal--Dual Mirror Descent}

\paragraph{Continuous-time dynamics.}
We combine the primal and dual mirror dynamics
\begin{equation}\label{eq:flow}
\begin{aligned}
x' &= y - x - \beta \nabla \phi^*(\nabla f(x)), \\
\bigl(\nabla \phi(y)\bigr)' &= \nabla \phi(x) - \nabla \phi(y) - \tfrac{1}{\mu}\nabla f(x).
\end{aligned}
\end{equation}
This flow preconditions the HNAG flow~\citep{chen_first_2019} through a primal--dual mirror coupling. Here $\mu>0$ is the relative convexity constant, so that $f-\mu\phi$ is convex, and $\beta>0$ controls the dual preconditioning strength. The $x$-dynamics contain a dual mirror step, while the $y$-dynamics follow a primal mirror update. The equilibrium of~\eqref{eq:flow} is $x=y=x^\star$.

To analyze stability, define the Lyapunov energy
\begin{equation}\label{eq:Lya}
E(z):=D_f(x,x^\star)+\mu D_\phi(x^\star,y), \qquad z=(x,y).
\end{equation}
The identity $D_\phi(x^\star,y)=D_{\phi^*}(\nabla\phi(y),\nabla\phi(x^\star))$ will be used to differentiate the second term.

\begin{lemma}[Energy dissipation]\label{lm:SL}
The energy~\eqref{eq:Lya} along the flow~\eqref{eq:flow} satisfies
\begin{equation}\label{eq:SL}
\begin{aligned}
\langle \nabla E(z),z'\rangle
={}&-E(z)-\beta\langle \nabla f(x),\nabla\phi^*(\nabla f(x))\rangle
-D_{f-\mu\phi}(x^\star,x)-\mu D_\phi(y,x).
\end{aligned}
\end{equation}
Consequently, if $f-\mu\phi$ is convex, then
$$
E(z(t))\le e^{-t}E(z(0)), \qquad t\ge 0.
$$
\end{lemma}

\begin{proof}
Let $\eta=\nabla\phi(y)$. Differentiating $E$ gives
$$
\partial_x E=\nabla f(x),
\qquad
\partial_\eta E
=\mu\bigl(\nabla\phi^*(\nabla\phi(y))-\nabla\phi^*(\nabla\phi(x^\star))\bigr)
=\mu(y-x^\star).
$$
Then by the chain rule, we obtain
$$
\begin{aligned}
\frac{\dd}{\dd t}E(z)
={}&\dual{\nabla E(z),z'}\\
={}&\dual{\nabla f(x),y-x-\beta\nabla\phi^*(\nabla f(x))}
+\mu\dual{y-x^\star,\nabla\phi(x)-\nabla\phi(y)-\tfrac1\mu\nabla f(x)}\\
={}&-\dual{\nabla f(x),x-x^\star}
-\beta\dual{\nabla f(x),\nabla\phi^*(\nabla f(x))}
+ \mu\dual{y-x^\star,\nabla\phi(x)-\nabla\phi(y)}
\end{aligned}
$$
which gives~\eqref{eq:SL} by the three-point identity~\eqref{eq:Bregmanidentity}. If $f-\mu\phi$ is convex, then the last two terms in~\eqref{eq:SL} are nonpositive. Hence $\frac{\dd}{\dd t}E(z(t))\le -E(z(t))$, and Gr\"onwall's inequality gives the result.
\end{proof}

\paragraph{Discretization.}
We discretize~\eqref{eq:flow} by the following implicit--explicit scheme:
\begin{subequations}\label{eq:scheme}
\begin{align}
x_{k+1}-x_k
&= \alpha\bigl(y_k-x_{k+1}-\beta\nabla\phi^*(\nabla f(x_k))\bigr),
\label{eq:scheme-x}\\
\nabla\phi(y_{k+1})-\nabla\phi(y_k)
&= \alpha\bigl(\nabla\phi(x_{k+1})-\nabla\phi(y_{k+1})\bigr)
-\tfrac{\alpha}{\mu}\nabla f(x_{k+1}).
\label{eq:scheme-y}
\end{align}
\end{subequations}
The gradient in the $x$-dynamics is treated explicitly in~\eqref{eq:scheme-x}, while $x_{k+1}$ and $y_{k+1}$ are treated implicitly to preserve the energy structure. Equivalently,~\eqref{eq:scheme-y} is a mirror-descent step, and $y_{k+1}$ solves
\begin{equation}\label{eq:yprox}
y_{k+1}\in\arg\min_y (1+\alpha)\phi(y)
-
\left\langle
\alpha\nabla\phi(x_{k+1})+\nabla\phi(y_k)-\tfrac{\alpha}{\mu}\nabla f(x_{k+1}),\,y
\right\rangle.
\end{equation}
We assume that $\phi$ is chosen so that the subproblem~\eqref{eq:yprox} can be computed efficiently.

\begin{lemma}[Descent identity]\label{lem:descentidentity}
The iterates of~\eqref{eq:scheme} satisfy
\begin{equation}\label{eq:descent}
\begin{split}
(1+\alpha)E(z_{k+1})-E(z_k)
={}&-D_f(x_k,x_{k+1})-\mu D_\phi(y_{k+1},y_k)\\
&+\alpha\dual{\nabla f(x_{k+1}),y_k-y_{k+1}}
-\alpha\beta\dual{\nabla f(x_{k+1}),\nabla\phi^*(\nabla f(x_k))}\\
&-\alpha D_{f-\mu\phi}(x^\star,x_{k+1})
-\alpha\mu D_\phi(y_{k+1},x_{k+1}).
\end{split}
\end{equation}
\end{lemma}
%
\begin{proof}
We present an outline here and refer to Appendix for details. The difference of Lyapunov functions 
$$
E(z_{k+1})-E(z_k) = \dual{\nabla E(z_{k+1}),z_{k+1}-z_k} -D_E(z_k,z_{k+1}).
$$
The first line on RHS of \eqref{eq:descent} is simply $-D_E(z_k,z_{k+1})$. 
Let $G$ denote the vector field in~\eqref{eq:flow}. The scheme~\eqref{eq:scheme} can be written as a correction of the implicit Euler step:
$$z_{k+1}-z_k=\alpha G(z_{k+1})+\binom{\alpha(y_k-y_{k+1})-\alpha\beta\bigl(\nabla\phi^*(\nabla f(x_k))-\nabla\phi^*(\nabla f(x_{k+1}))\bigr)}{0}.$$
We follow Lemma \ref{lm:SL} to calculate $\dual{\nabla E(z_{k+1}), \alpha G(z_{k+1})}$. The discrepancy terms are in the second and third lines in \eqref{eq:descent}. 
%
%
\end{proof}

\section{Adaptive Accelerated Mirror Descent}

We now introduce AAMD, which adaptively updates $(\alpha_k,\beta_k,\mu_k)$ by monitoring the slack in the Lyapunov dissipation inequality.

\paragraph{Accumulated Stability and Reduced Line Search.}
We first consider the case where $\mu>0$ is known and rewrite the remaining parameters as $(\alpha_k,L_k)$ with $1/L_k=\alpha_k\beta_k$.


\begin{lemma}[Descent inequality]\label{lem:descent}
The iterates of~\eqref{eq:scheme} satisfy
\begin{equation}\label{eq:descentbound}
E(z_{k+1})\le \frac{1}{1+\alpha_k}E(z_k)+b_k,
\qquad
b_k=b_k^{(1)}+b_k^{(2)}+b_k^{(3)},
\end{equation}
where
\begin{subequations}\label{eq:residuals}
\begin{align}
b_k^{(1)}
&=
\frac{1}{L_k}D_{\phi^*}(\nabla f(x_{k+1}),\nabla f(x_k))
-
D_f(x_k,x_{k+1}),
\\
b_k^{(2)}
&=
\alpha_k\dual{\nabla f(x_{k+1}),y_k-y_{k+1}}
-
\frac{1}{L_k}D_{\phi^*}(\nabla f(x_{k+1}),0)
-
\mu D_\phi(y_{k+1},y_k),
\\
b_k^{(3)}
&=
-\frac{1}{L_k}D_{\phi^*}(0,\nabla f(x_k))
-
\alpha_k\mu D_\phi(y_{k+1},x_{k+1}).
\end{align}
\end{subequations}
\end{lemma}

\begin{proof}
Let $g_k=\nabla f(x_k)$. Dropping the negative terms $-\alpha D_{f-\mu\phi}(x^\star,x_{k+1})$ in the last line of~\eqref{eq:descent} and applying the three-point identity gives
$$
-\dual{g_{k+1},\nabla\phi^*(g_k)}
=
-D_{\phi^*}(g_{k+1},0)
-D_{\phi^*}(0,g_k)
+
D_{\phi^*}(g_{k+1},g_k).
$$
Rearranging the terms yields~\eqref{eq:descentbound}.
\end{proof}

The resulting algorithm is summarized in Algorithm~\ref{alg:A2MD}. We use a modified adaptive backtracking strategy from~\cite{cavalcanti2024adaptive}. The perturbation $b_k$ consists of three computable parts. The term $b_k^{(3)}$ is always nonpositive and acts as a stability buffer generated by the dual mirror step. When $p_k>0$, the current parameters fail to satisfy the descent condition, and we distinguish two cases. 

If $b_k^{(1)}>0$, then the local smoothness estimate $L_k$ is too small. We update
$$
L_k
\gets
\max\left\{
c_1L_k,\,
\frac{D_{\phi^*}(g_{k+1},g_k)}{D_f(x_k,x_{k+1})}
\right\},
\qquad
g_k=\nabla f(x_k).
$$

If $b_k^{(2)}>0$, then the stepsize $\alpha_k$ is too large. We decrease $\alpha_k$ by
$$
\alpha_k
\gets
\min\left\{
\alpha_k/c_2,\,
\frac{
\frac{1}{L_k}D_{\phi^*}(g_{k+1},0)
+\mu D_\phi(y_{k+1},y_k)
}{
\dual{g_{k+1},y_k-y_{k+1}}
}
\right\},
$$
until $p_k\le 0$. Here $c_1,c_2>1$ are fixed backtracking constants. The second formula is obtained by enforcing $b_k^{(2)}=0$. Thus the line search either increases the local curvature estimate or decreases the acceleration parameter, guaranteeing termination.

After an admissible pair $(L_k,\alpha_k)$ is found, we use the spectral estimates in Lines~13--14 as the initial guess for the next iteration. The ratio defining $L_{k+1}$ measures the local relative smoothness along the current step, while $\alpha_{k+1}=\sqrt{\mu/L_{k+1}}$ gives the corresponding accelerated parameter.

\begin{algorithm}[t]
\caption{AAMD (Adaptive Accelerated Mirror Descent)}
\label{alg:A2MD}
\begin{spacing}{1.15}
\begin{algorithmic}[1]
\STATE \textbf{Input:} $x_0,y_0\in\mathbb R^n$, $\mu>0$, and maxIt
\STATE \textbf{Set} $L_0=1$, $\alpha_0=1$, $p_{-1}=0$
\FOR{$k=0,1,\dots,$ maxIt}
\REPEAT
\STATE $x_{k+1}=\frac{1}{1+\alpha_k}\left(x_k+\alpha_ky_k-\frac{1}{L_k}\nabla\phi^*(\nabla f(x_k))\right)$
\STATE $\eta_{k+1}=\frac{1}{1+\alpha_k}\left(\nabla\phi(y_k)+\alpha_k\nabla\phi(x_{k+1})-\frac{\alpha_k}{\mu}\nabla f(x_{k+1})\right)$
\STATE $y_{k+1}=\nabla\phi^*(\eta_{k+1})$
\STATE $p_k=\frac{1}{1+\alpha_k}\left(p_{k-1}+\sum_{i=1}^3 b_k^{(i)}\right)$
\IF{$p_k>0$}
\STATE update $L_k$ and $\alpha_k$ via adaptive line search
\ENDIF
\UNTIL{$p_k\le 0$}
\STATE $L_{k+1}=D_{\phi^*}(\nabla f(x_{k+1}),\nabla f(x_k))/D_f(x_k,x_{k+1})$
\STATE $\alpha_{k+1}=\sqrt{\mu/L_{k+1}}$
\ENDFOR
\end{algorithmic}
\end{spacing}
\end{algorithm}

The global stability is controlled by the following multi-step estimate.

\begin{theorem}[Multi-step convergence]\label{thm:accumulative}
Assume $f-\mu\phi$ is convex with $\mu>0$, and let $\{z_k\}$ be generated by Algorithm~\ref{alg:A2MD}. Define the accumulated perturbation by
$$
p_{-1}=0,
\qquad
p_k=\frac{1}{1+\alpha_k}(p_{k-1}+b_k).
$$
Then the Lyapunov energy~\eqref{eq:Lya} satisfies
$$
E(z_{k+1})
\le
\left(\prod_{i=0}^k\frac{1}{1+\alpha_i}\right)E(z_0).
$$
\end{theorem}

\begin{proof}
By induction,
$$
E(z_{k+1})
\le
\left(\prod_{i=0}^k\frac{1}{1+\alpha_i}\right)E(z_0)
+
p_k.
$$
Since the algorithm enforces $p_k\le 0$ for all $k$, the result follows.
\end{proof}

Unlike standard line search methods, AAMD triggers backtracking only when $p_k$ becomes positive. In practice, the negative buffer term $b_k^{(3)}$ is often large enough that line search is needed only during the initial transient phase. By continuity, the parameters $(L_{k+1},\alpha_{k+1})$ are usually good initial guesses for the next iteration. In our numerical experiments, fewer than ten backtracking steps occur in total, all during the initial transient phase.

\begin{remark}
In all numerical experiments, we use $c_1=2$ and $c_2=1.5$. The method is not sensitive to these values, since they only serve as safeguards for line-search termination.
\end{remark}

\paragraph{General convexity via homotopy.}
For non-strongly convex objectives, i.e., $\mu=0$, we use a homotopy strategy that solves a sequence of $\varepsilon$-perturbed problems. This variant, denoted by \textbf{AAMD-0}, is summarized in Algorithm~\ref{alg:AAMD-0}.

\begin{algorithm}[t]
\caption{AAMD-0}\label{alg:AAMD-0}
\begin{spacing}{1.15}
\begin{algorithmic}[1]
\STATE Initialize $x_0,y_0\in\mathbb R^n$, $L_0=\alpha_0=\varepsilon_0=1$, $m=m_0$, $s=0$, $k_0=0$
\FOR{$k=0,1,2,\dots$}
\STATE Apply Algorithm~\ref{alg:A2MD} with parameter $\varepsilon_s$ for one step:
$(x_{k+1},y_{k+1})={\rm AAMD}(x_k,y_k,\varepsilon_s,1)$
\IF{$E_k\le E_{k_s}/2$ \textbf{or} $k\ge k_s+m$}
\STATE $\varepsilon_{s+1}\gets \varepsilon_s/2$, \quad $m\gets \lfloor\sqrt{2}\,m\rfloor+1$
\STATE $k_{s+1}\gets k$, \quad $s\gets s+1$
\ENDIF
\ENDFOR
\end{algorithmic}
\end{spacing}
\end{algorithm}

The key idea is simple. For a fixed $\varepsilon$, the inner iterations reduce the error to the scale $O(R\varepsilon)$, where $R$ controls the Bregman distance of the iterates:
\begin{equation}\label{eq:boundedness}
D_\phi(x^\star,x_k)\le \frac12 R^2,\qquad \forall k\ge 0.
\end{equation}
When $\varepsilon$ is halved, the number of required inner iterations increases by a factor of $\sqrt{2}$. This schedule preserves the overall $O(1/k^2)$ complexity. 

We use two criteria for halving the perturbation. The first is the decay condition
$$
E_k\le E_{k_s}/2
$$
which may be met before the prescribed $m$ inner steps when the problem is strongly convex. The second is the iteration cap $k\ge k_s+m$. Since $E_k$ is not directly observable, we use the computable stopping test
$$
\|\nabla f(x_k)\|^2\leq \|\nabla f(x_{k_s})\|^2/2.
$$
With this rule, AAMD-0 works for both $\mu=0$ and $\mu>0$ without knowing $\mu$ in advance.

\begin{theorem}\label{thm:homotopy}
Assume $f$ and $Lf^*-\phi^*$ are convex. Let $\{x_k\}$ be generated by Algorithm~\ref{alg:AAMD-0}, and assume the boundedness condition~\eqref{eq:boundedness}. Let $k_s$ be the total number of steps after $\varepsilon$ has been halved exactly $s$ times, so that $\varepsilon=2^{-s}\varepsilon_0$. Then there exists a constant $C>0$ such that
$$
\frac{E_{k_s}}{E_0}
\le
\frac{R^2+1}{\left(Ck_s+\varepsilon_0^{-1/2}\right)^2}
=
\mathcal O\left(\frac{1}{k_s^2}\right).
$$
Consequently, $\mathcal O(\sqrt{1/{\rm tol}})$ iterations suffice to achieve $E_{k_s}/E_0\le {\rm tol}$.
\end{theorem}

We defer the proof to Appendix~\ref{sec:homotopy_proof}. The proof also requires a structural lower bound on the adaptive parameters to prevent $\alpha_k$ from becoming much smaller than $2/k$, which would destroy acceleration.

\paragraph{On the boundedness assumption.}
The boundedness assumption \eqref{eq:boundedness} in Theorem~\ref{thm:homotopy} is needed to convert the homotopy decrease for the regularized problems into an $O(1/k^2)$ rate for the original convex problem. It holds, for example, when the relevant Lyapunov sublevel set is bounded, when $\phi$ is coercive on the iterates, or when the feasible set is bounded. Without such a condition, a global homotopy guarantee cannot be expected. 

In Appendix~\ref{sec:homotopy_proof}, we show that this boundedness condition holds for the continuous flow~\eqref{eq:flow} and for its discretization when $\alpha$ is sufficiently small. In practice, boundedness can be enforced by projection, constraints, or a trust-region safeguard.

\paragraph{Convergence of iterates.}
The preceding results give convergence of the objective value and Lyapunov energy. If the minimizer $x^\star$ is unique and $\phi$ is strongly convex on the relevant sublevel set, then $D_\phi(x^\star,y_k)\to0$ implies $y_k\to x^\star$. The update equations and the vanishing energy then give $x_k\to x^\star$. Without uniqueness, the results imply convergence of the function values and the Bregman distance to the solution set, but not strong convergence of $\|x_k-x^\star\|$ to a fixed minimizer.

\paragraph{Extension to composite problems.}
AAMD also extends to composite problems of the form $\min_x f(x)+g(x)$, where $f$ is smooth and $g$ is convex but possibly nonsmooth. The $x$-update is replaced by a proximal gradient step adapted to the mirror geometry:
\begin{equation}\label{eq:compositeproxsubproblem}
  \begin{aligned}
x_{k+1} = \arg\min_{x\in\mathbb{R}^n} \left \{\frac{1}{L_k(1+\alpha_k)}\phi\left(-L_k(1+\alpha_k)\left(x-z_{k+1}\right)\right)+\dual{\nabla f(x_{k}),x}+g(x)\right \}
  \end{aligned}
\end{equation}
where $z_{k+1} = \frac{1}{1+\alpha_k}(x_k+\alpha_k y_k)$. The $y$-update uses the subgradient computed in the $x$-update step. The linesearch procedure for $L_k$ and the spectral update remain the same as in the smooth case. For $\mu=0$, we employ the same homotopy strategy as in the smooth case in Appendix \ref{sec:homotopy_proof}.

This allows the method to handle nonsmooth penalties and simple constraints. The resulting algorithm, AAproxMD, is summarized in Algorithm~\ref{alg:AAproxMD}. Its convergence analysis is given in Appendix~\ref{sec:composite}.

\begin{example}[LASSO problem]
Consider the LASSO problem
\begin{equation}
    \min_{x \in \mathbb{R}^d} F(x) := \frac{1}{2} \|Ax - b\|^2 + \lambda \|x\|_1, \quad \phi(x)=\frac{1}{2} x^\top D x,
\end{equation}
where $A \in \mathbb{R}^{n \times d}$ with $n<d$ is row full rank and $D=\mathrm{diag}(A^\top A)$. Since $D$ is diagonal and positive definite, the subproblem described in Algorithm \ref{alg:AAproxMD} admits a closed-form solution given by a generalized soft-thresholding operator. More details are provided in Appendix \ref{ex: Lasso}.
\end{example}

\section{Numerical Experiments}

The experiments are intended as proof-of-concept tests for the theory rather
than as a broad empirical benchmark.  We choose problems that satisfy relative
smoothness or dual relative smoothness and for which the effect of nonlinear
preconditioning, acceleration, and adaptivity can be isolated.  The current
tests therefore focus on deterministic convex problems.  Large-scale stochastic
and nonconvex optimization are outside the scope of this paper.


In these examples, AAMD is competitive with, and often faster than, the
tested baselines.

\paragraph{Regularized logistic regression}\label{sec:logistic}
Consider the regularized logistic regression problem with a \textit{symmetrized logistic mirror regularizer}
\begin{equation}\label{eq:logisticregression}
    \min_{x \in \mathbb{R}^d} f(x) := \frac{1-\mu}{n} \sum_{i=1}^n \ln(1 + \exp(-b_i a_i^\top x)) + \mu \phi(x),
\end{equation}
where $\{(a_i,b_i)\}_{i=1}^n$ are the data-label pairs with $a_i \in [-1,1]^d$ and $b_i \in \{-1,1\}$, $\mu<\max_{i=1,\cdots,n}\|a_i\|^2$ and $\phi(x)=2\sum_{j=1}^n \ln(1+\exp(x_j))-x_j$. Then $f$ is $\mu$-relatively strongly convex with respect to $\phi$~\cite{lu2018relatively}. As $f$ is a pointwise average of single LogSumExp functions $\ln(1+\exp(-b_i a_i^\top x))$ and $\phi$, by Proposition 4.9 \cite{laude2025anisotropic}, the dual smoothness constant $L=\max_{i=1,\cdots,n}\|a_i\|^2$.

We define $(a_i,b_i)$ on two datasets: (1) the Adult Census Income dataset. After normalizing and removing entries with missing values, the dataset contains
30,162 samples and 14 features. The Lipschitz constant is $4.84$, and the dual relative smoothness constant is $3.31$. Set the regularization parameter $\mu=0.1$. The results are shown in Fig.~\ref{fig:logistic_time}(a). 

(2) the mushroom dataset. After normalizing, the dataset contains $8,125$ samples and $139$ features. The Lipschitz constant is $75.93$, and the dual relative smoothness constant is $10.61$. Set the regularization parameter $\mu=0.3$. The results are shown in Fig.~\ref{fig:logistic_time}(b).

We compare AAMD and AAMD-0 with Adaptive Accelerated Gradient Descent methods (A$^2$GD)~\cite{xu2026adaptiveacceleratedgradientdescent}, anisotropic proximal gradient (anisoPG)~\cite{laude2025anisotropic,maddison2021dual}, anisotropic proximal gradient with line search (LSanisoPG)~\cite{laude2025anisotropic}, and Nesterov's accelerated gradient (NAG) method~\cite{nesterov1983method}. Since the problem is strongly convex on any compact domain, all curves show a linear convergence pattern. On both datasets, AAMD outperforms the tested baselines by a large margin.

\begin{figure}[htp]
\centering
\begin{minipage}{0.37\textwidth}
Comparing the performance curves of A$^2$GD, NAG, and anisoPG, we observe that methods using acceleration alone (NAG) or preconditioning alone (anisoPG) are consistently outperformed by A$^2$GD, highlighting the critical role of combining adaptivity and acceleration.

\smallskip

The superior performance of LSanisoPG over A$^2$GD further demonstrates the importance of combining preconditioning with adaptivity.
\end{minipage}
\hfill
\begin{minipage}{0.6\textwidth}
\centering
\includegraphics[width=0.49\textwidth]{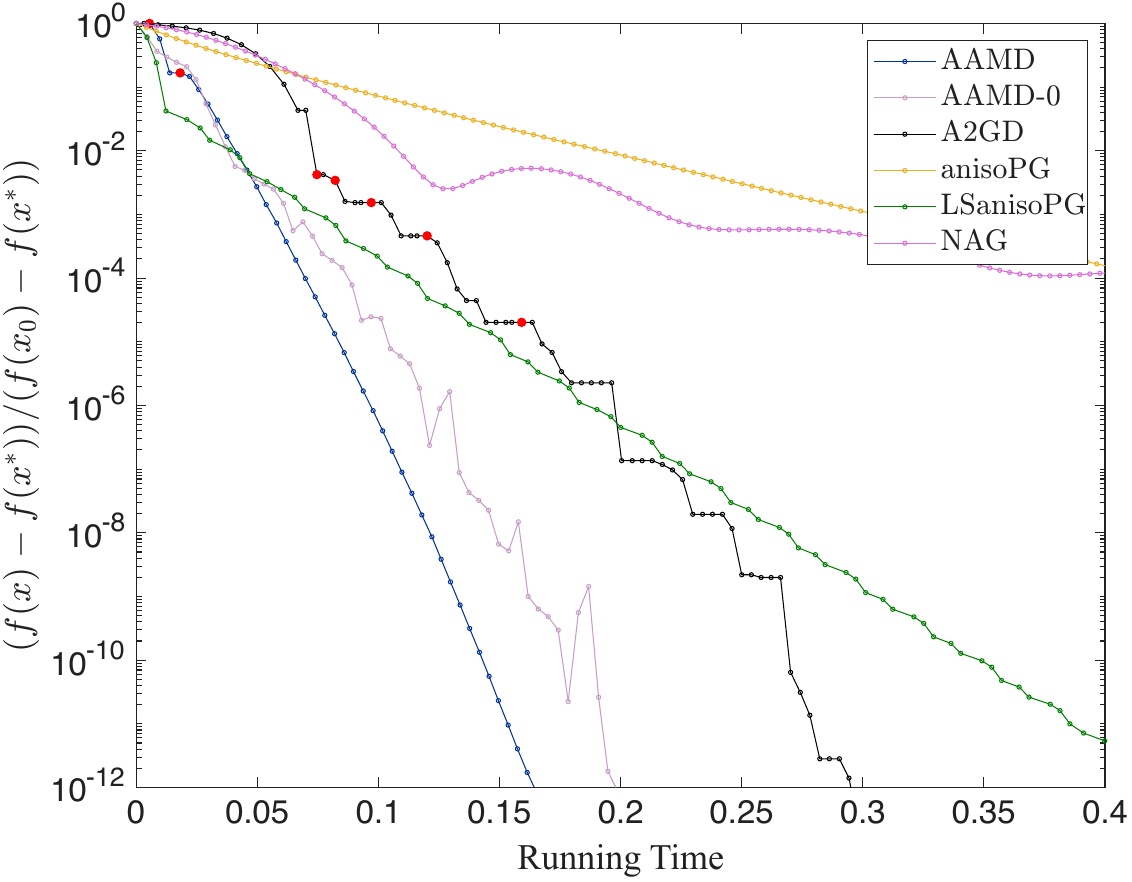}
\hfill
\includegraphics[width=0.49\textwidth]{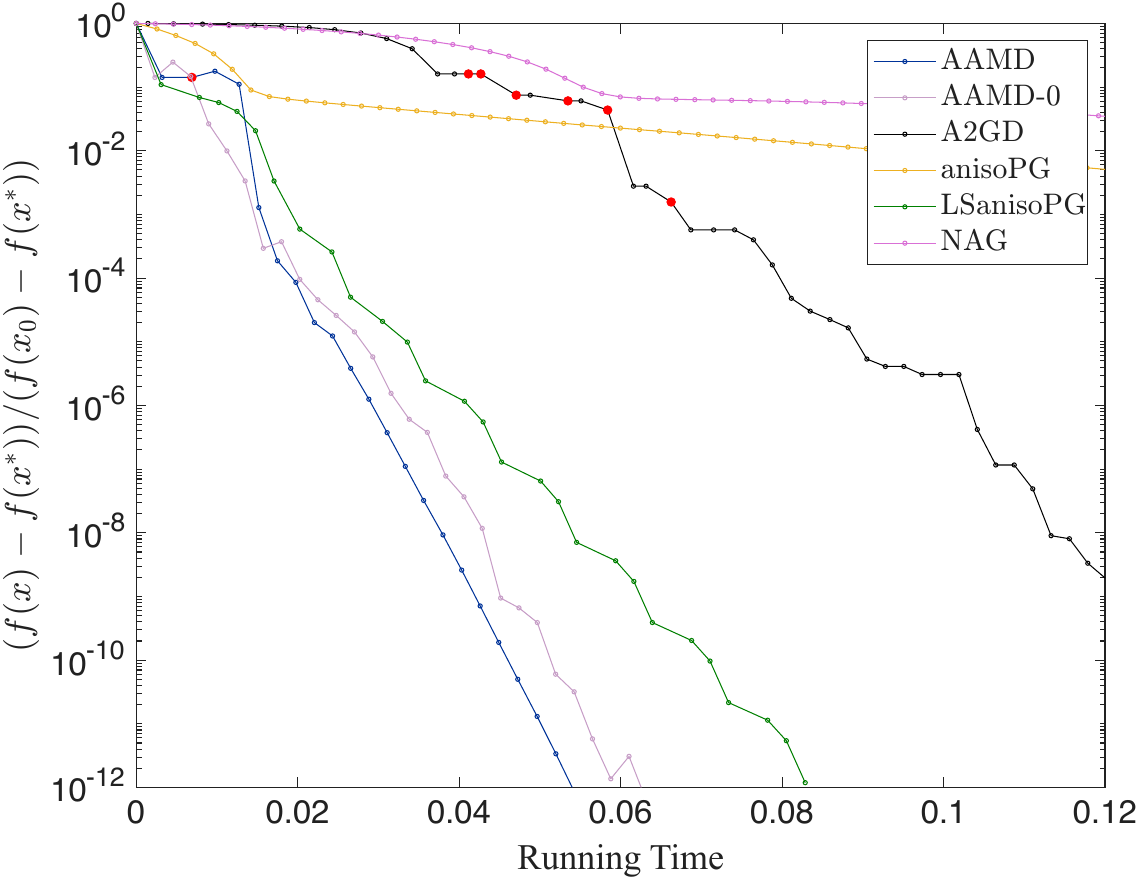}

\vspace{1mm}
\makebox[0.48\textwidth]{\small (a) Adult Census Income}
\hfill
\makebox[0.48\textwidth]{\small (b) Mushroom}

\caption{Log relative error vs. execution time. Red dots indicate the gradient steps incurred by line search.}
\label{fig:logistic_time}
\end{minipage}
\end{figure}

Finally, the fact that AAMD outperforms LSanisoPG and A$^2$GD indicates that incorporating acceleration on top of both preconditioning and adaptivity yields additional and substantial gains.

The red dots in Figures~\ref{fig:logistic_time}--\ref{fig:Lk_compare} mark all backtracking steps. In these experiments, fewer than ten backtracking steps occur in total, all during the initial transient phase. Once the local estimates of $L_k$ and $\alpha_k$ stabilize, the method proceeds without further line search. This supports the role of the accumulated perturbation budget: line search acts mainly as a safeguard, not as a repeated per-iteration procedure.

\begin{figure}[htp]
\centering
\begin{minipage}{0.45\textwidth}
We demonstrate the benefit of adaptively estimating the dual smoothness constant $L$ and the learning rate $\alpha$ in AAMD, compared with methods such as anisoPG that use a fixed dual smoothness constant. Fig.~\ref{fig:Lk_compare} shows the estimates of $L_k$ and $\alpha_k$ over the iterations. Both AAMD and LSanisoPG capture smaller values of $L_k$ than the true dual smoothness constant, shown in yellow, which leads to faster convergence. Although LSanisoPG obtains smaller $L_k$, the adaptive line-search strategy in AAMD gives a more stable optimization process. For $\alpha_k$, both AAMD and AAMD-0 choose larger step sizes than the default accelerated rate $\sqrt{\mu/L}$. 

\end{minipage}
\hfill
\begin{minipage}{0.52\textwidth}
\centering
\includegraphics[width=0.48\textwidth]{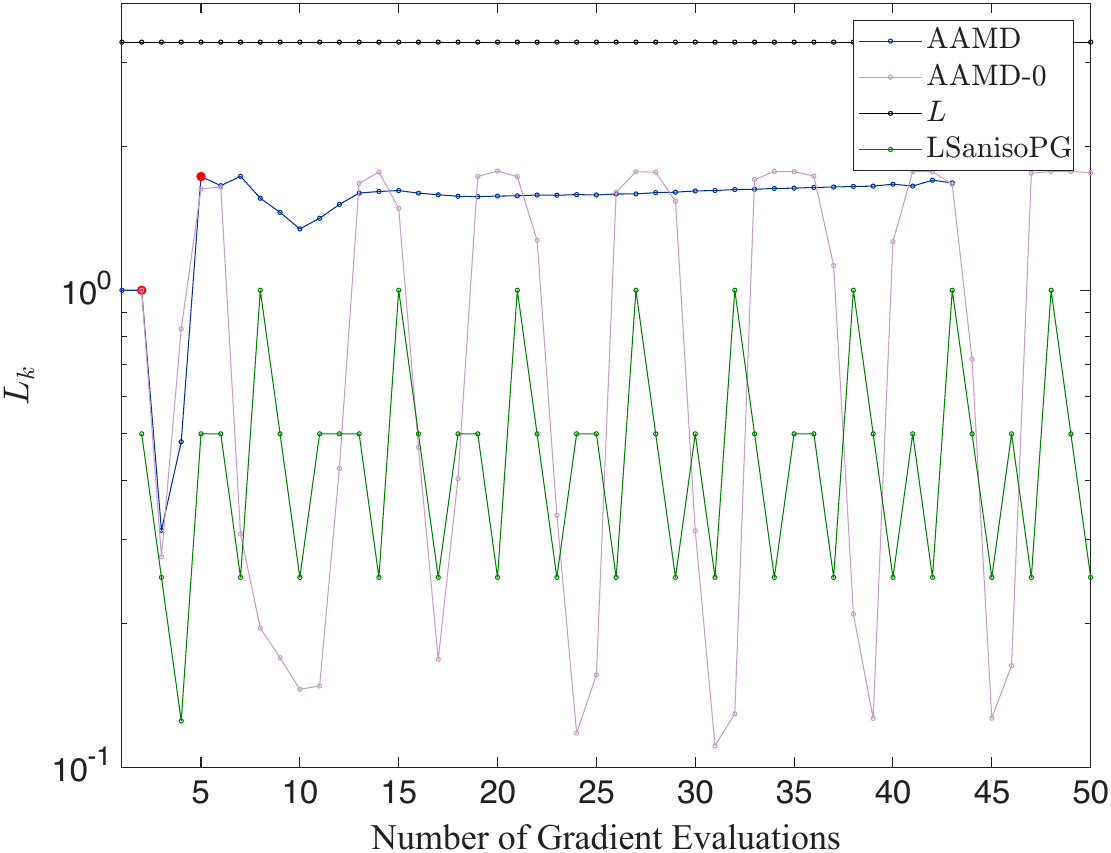}
\hfill
\includegraphics[width=0.48\textwidth]{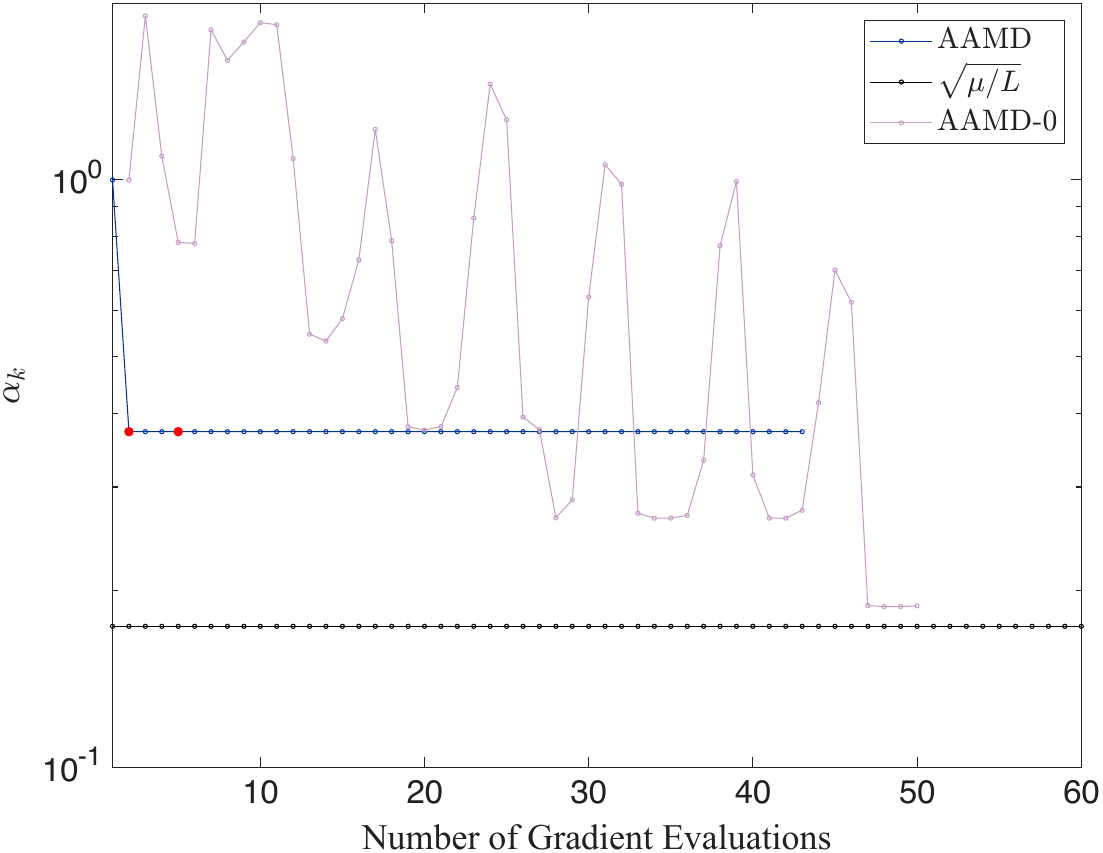}
\caption{Estimation of $L_k$ and $\alpha_k$ vs. number of gradient evaluations on the regularized logistic regression task. Red dots indicate the gradient steps incurred by line search.}
\label{fig:Lk_compare}
\end{minipage}
\end{figure}


\paragraph{Polynomial-of-norm objectives.}

In the experiment, we use the quartic benchmark from~\cite{lu2018relatively}:
\begin{equation}\label{eq:quartic-objective}
f(x)=\frac14\|Ax-b\|^4+\frac13\|Bx\|^3+\frac12\|Cx\|^2,
\end{equation}
where $A,B,C$ are positive definite matrices. 
Thus $f$ is dual relatively smooth and strongly convex with respect to
\begin{equation}\label{eq:polytonorm-mirror}
\phi(x)=\frac14\|x\|^4+\frac12\|x\|^2.
\end{equation}
We set $n=2048$. We draw $A_0\sim\mathcal N(\mathbf 0,I_{n\times n})$ and set $A=A_0A_0^\top/n$. The matrices $B$ and $C$ are generated in the same way. We take $b=0$, initialize $x_0$ with independent entries from $\operatorname{Unif}(0,0.1)$, and set $y_0=x_0$. The minimizer is $x^\star=0$.

\citet{lu2018relatively} showed that $f$ is relatively smooth and strongly convex with respect to $\phi$, verifying (A1). Here we also use the dual relative smoothness and strong convexity condition (A1$^*$). Proposition~\ref{prop:dual_relative_smoothness} in Appendix~\ref{section_polynomialnorm_proof} shows that, for polynomial-of-norm mirror maps, the highest-order term controls the dual relative smoothness at infinity, while the lowest-order term controls the local dual behavior near the origin. Hence the mirror map~\eqref{eq:polytonorm-mirror}, which matches the highest and lowest powers in~\eqref{eq:quartic-objective}, captures the geometry over the whole space. Although such constants exist, they are not needed for running AAMD, since the line search adapts to the local geometry.

To compute $\nabla\phi^*(\cdot)$, we use the method in~\cite{lu2018relatively}. For~\eqref{eq:polytonorm-mirror}, this reduces to solving a one-dimensional monotone equation for the radial variable.

We compare AAMD with mirror descent (MD), accelerated Bregman proximal gradient (ABPG)~\cite{hanzely_accelerated_2018}, anisotropic proximal gradient with line search (LSanisoPG)~\cite{laude2025anisotropic}, and Nesterov's accelerated gradient method (NAG)~\cite{nesterov1983method}. For MD and ABPG, the relative smoothness constant is computed as in~\cite{lu2018relatively}. For NAG, we estimate the Euclidean smoothness constant by
$\displaystyle
L_f=\bigl(3\|A\|^4+2\|B\|^3+\|C\|^2\bigr)\|x_0\|_2^2,
$
which can be large.

\vskip -10pt
\begin{figure}[htp]
\centering
\begin{minipage}{0.5\textwidth}
The results are shown in Fig.~\ref{fig:polynomialnorm_time}. Since the problem is not strongly convex, all methods show sublinear convergence. AAMD outperforms the tested baselines. It incurs extra gradient evaluations from line search only at the beginning, due to the naive initialization $L_0=1$. The function-gap curve has mild oscillations because the function gap captures only part of the Lyapunov energy. These oscillations could be removed by adding the monotonicity safeguard $f(x_{k+1})\le f(x_k)$, but we do not include this safeguard in AAMD so that the effect of the adaptive Lyapunov budget is shown directly.

\end{minipage}
\hfill
\begin{minipage}{0.45\textwidth}
\centering
\includegraphics[width=0.85\textwidth]{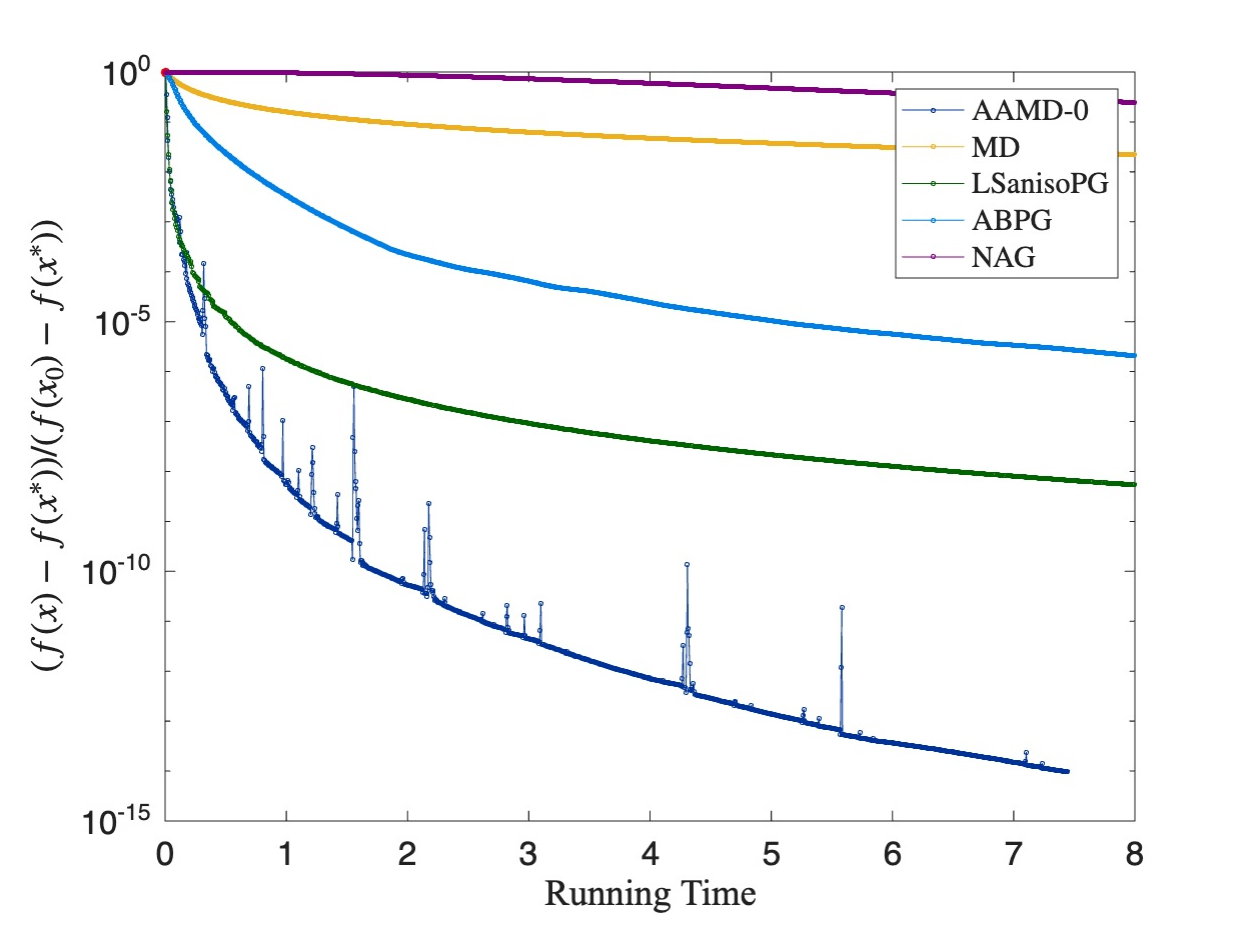}
\caption{Log relative error vs. execution time on the quartic objective. Red dots indicate the gradient steps incurred by line search.}
\label{fig:polynomialnorm_time}
\end{minipage}
\end{figure}

\bibliographystyle{plainnat}
\bibliography{references}

@article{OikonoQuanLaudePatrin2025Nonlinearly,
	author = {Oikonomidis, Konstantinos and Quan, Jan and Laude, Emanuel and Patrinos, Panagiotis},
	journal = {arXiv preprint arXiv:2502.08532},
	title = {Nonlinearly Preconditioned Gradient Methods under Generalized Smoothness},
	year = {2025}}

@misc{amsmath,
	author = {{American Mathematical Society}},
	title = {User's Guide for the \texttt{amsmath} Package (Version 2.0)},
	url = {ftp://ftp.ams.org/pub/tex/doc/amsmath/amsldoc.pdf},
	urldate = {2015-07-30},
	year = 2002}

@misc{pgfplots,
	author = {Christian Feuers\"anger},
	month = may,
	title = {Manual for Package \texttt{PGFPLOTS}},
	url = {http://sourceforge.net/projects/pgfplots},
	year = 2015}

@article{lu2018relatively,
	author = {Lu, Haihao and Freund, Robert M. and Nesterov, Yurii},
	doi = {https://doi.org/10.1137/16M1099546},
	journal = {SIAM Journal on Optimization},
	number = {1},
	pages = {333-354},
	title = {Relatively Smooth Convex Optimization by First-Order Methods, and Applications},
	volume = {28},
	year = {2018}}

@article{chen1993convergence,
	author = {Chen, Gong and Teboulle, Marc},
	doi = {https://doi.org/10.1137/0803026},
	eprint = {https://doi.org/10.1137/0803026},
	journal = {SIAM Journal on Optimization},
	number = {3},
	pages = {538-543},
	title = {Convergence Analysis of a Proximal-Like Minimization Algorithm Using {Bregman} Functions},
	volume = {3},
	year = {1993}}

@inproceedings{krichene2015accelerated,
	author = {Krichene, Walid and Bayen, Alexandre and Bartlett, Peter L},
	booktitle = {Advances in Neural Information Processing Systems},
	editor = {C. Cortes and N. Lawrence and D. Lee and M. Sugiyama and R. Garnett},
	publisher = {Curran Associates, Inc.},
	title = {Accelerated Mirror Descent in Continuous and Discrete Time},
	url = {https://proceedings.neurips.cc/paper_files/paper/2015/file/f60bb6bb4c96d4df93c51bd69dcc15a0-Paper.pdf},
	volume = {28},
	year = {2015}}

@article{hanzely_accelerated_2018,
	author = {Filip Hanzely and Peter Richt{\'a}rik and Lin Xiao},
	journal = {Computational Optimization and Applications},
	pages = {405 - 440},
	title = {Accelerated Bregman proximal gradient methods for relatively smooth convex optimization},
	url = {https://api.semanticscholar.org/CorpusID:52585212},
	volume = {79},
	year = {2018}}

@misc{chen_first_2019,
	archiveprefix = {arXiv},
	author = {Long Chen and Hao Luo},
	eprint = {1912.09276},
	primaryclass = {math.OC},
	title = {First order optimization methods based on Hessian-driven Nesterov accelerated gradient flow},
	url = {https://arxiv.org/abs/1912.09276},
	year = {2019}}

@article{yuan2023analysis,
	author = {Yuan, Ya-Xiang and Zhang, Yi},
	date = {2024/07/08},
	doi = {https://doi.org/10.1007/s40305-024-00542-3},
	id = {Yuan2024},
	isbn = {2194-6698},
	journal = {Journal of the Operations Research Society of China},
	title = {Analyze Accelerated Mirror Descent via High-Resolution {ODE}s},
	year = {2024}}

@article{Hanzely2021,
	author = {Hanzely, Filip and Richt{\'a}rik, Peter and Xiao, Lin},
	date = {2021/06/01},
	doi = {10.1007/s10589-021-00273-8},
	id = {Hanzely2021},
	isbn = {1573-2894},
	journal = {Computational Optimization and Applications},
	number = {2},
	pages = {405--440},
	title = {Accelerated Bregman proximal gradient methods for relatively smooth convex optimization},
	url = {https://doi.org/10.1007/s10589-021-00273-8},
	volume = {79},
	year = {2021}}

@book{nemirovskij1983problem,
	author = {Nemirovskij, Arkadij Semenovi{\v{c}} and Yudin, David Borisovich},
	isbn = {9780471103455},
	lccn = {82011065},
	publisher = {Wiley},
	series = {A Wiley-Interscience publication},
	title = {Problem Complexity and Method Efficiency in Optimization},
	year = {1983}}

@article{nesterov2005smooth,
	author = {Nesterov, Yurii},
	date = {2005/05/01},
	doi = {https://doi.org/10.1007/s10107-004-0552-5},
	id = {Nesterov2005},
	isbn = {1436-4646},
	journal = {Mathematical Programming},
	number = {1},
	pages = {127--152},
	title = {Smooth minimization of non-smooth functions},
	volume = {103},
	year = {2005}}

@misc{chen2025accelerated,
	archiveprefix = {arXiv},
	author = {Long Chen and Luo Hao and Jingrong Wei},
	eprint = {2505.04065},
	primaryclass = {math.OC},
	title = {Accelerated Gradient Methods Through Variable and Operator Splitting},
	url = {https://arxiv.org/abs/2505.04065},
	year = {2025}}

@article{kim2023mirror,
	author = {Kim, Jaeyeon and Park, Chanwoo and Ozdaglar, Asuman and Diakonikolas, Jelena and Ryu, Ernest K},
	journal = {arXiv preprint arXiv:2311.17296},
	title = {Mirror duality in convex optimization},
	year = {2023}}

@article{DragomTaylordAspreBolte2022Optimal,
	author = {Dragomir, Radu-Alexandru and Taylor, Adrien B. and d'Aspremont, Alexandre and Bolte, J{\'e}r{\^o}me},
	date = {2022/07/01},
	doi = {https://doi.org/10.1007/s10107-021-01618-1},
	id = {Dragomir2022},
	isbn = {1436-4646},
	journal = {Mathematical Programming},
	number = {1},
	pages = {41--83},
	title = {Optimal complexity and certification of {Bregman} first-order methods},
	volume = {194},
	year = {2022}}

@article{nesterov1983method,
	author = {Nesterov, Yurii},
	journal = {Doklady Akademii Nauk SSSR},
	number = {3},
	pages = {543--547},
	title = {A method of solving a convex programming problem with convergence rate $O\bigl(\frac1{k^2}\bigr)$},
	volume = {269},
	year = {1983}}

@article{nemirovski2004prox,
	author = {Nemirovski, Arkadi},
	doi = {10.1137/S1052623403425629},
	journal = {SIAM Journal on Optimization},
	number = {1},
	pages = {229-251},
	title = {Prox-Method with Rate of Convergence O(1/t) for Variational Inequalities with Lipschitz Continuous Monotone Operators and Smooth Convex-Concave Saddle Point Problems},
	volume = {15},
	year = {2004}}

@book{bertsekas2009convex,
	author = {Bertsekas, D.},
	isbn = {9781886529311},
	publisher = {Athena Scientific},
	series = {Athena Scientific optimization and computation series},
	title = {Convex Optimization Theory},
	year = {2009}}

@book{rockafellar1970convex,
	author = {R. Tyrrell Rockafellar},
	isbn = {9780691015866},
	publisher = {Princeton University Press},
	title = {Convex Analysis},
	url = {http://www.jstor.org/stable/j.ctt14bs1ff},
	urldate = {2025-12-25},
	year = {1970}}

@article{laude2025anisotropic,
	author = {Laude, Emanuel and Patrinos, Panagiotis},
	journal = {Mathematical Programming},
	pages = {1--45},
	publisher = {Springer},
	title = {Anisotropic proximal gradient},
	year = {2025}}

@misc{chen2025accmd,
	archiveprefix = {arXiv},
	author = {Long Chen and Hao Luo and Jingrong Wei and Zeyi Xu and Yuan Yao},
	eprint = {2601.19038},
	primaryclass = {math.OC},
	title = {Accelerated Mirror Descent Method through Variable and Operator Splitting},
	url = {https://arxiv.org/abs/2601.19038},
	year = {2026}}

@article{malitsky2024adaptive,
	author = {Malitsky, Yura and Mishchenko, Konstantin},
	journal = {Advances in Neural Information Processing Systems},
	pages = {100670--100697},
	title = {Adaptive proximal gradient method for convex optimization},
	volume = {37},
	year = {2024}}

@article{malitsky2019adaptive,
	author = {Malitsky, Yura and Mishchenko, Konstantin},
	journal = {arXiv preprint arXiv:1910.09529},
	title = {Adaptive gradient descent without descent},
	year = {2019}}

@misc{xu2026adaptiveacceleratedgradientdescent,
	archiveprefix = {arXiv},
	author = {Zeyi Xu and Long Chen},
	eprint = {2601.19013},
	primaryclass = {math.OC},
	title = {Adaptive Accelerated Gradient Descent Methods for Convex Optimization},
	url = {https://arxiv.org/abs/2601.19013},
	year = {2026}}

@article{cavalcanti2024adaptive,
	author = {Cavalcanti, Joao V and Lessard, Laurent and Wilson, Ashia C},
	journal = {arXiv preprint arXiv:2408.13150},
	title = {Adaptive Backtracking Line Search},
	year = {2024}}

@article{maddison2021dual,
  title={Dual space preconditioning for gradient descent},
  author={Maddison, Chris J and Paulin, Daniel and Teh, Yee Whye and Doucet, Arnaud},
  journal={SIAM Journal on Optimization},
  volume={31},
  number={1},
  pages={991--1016},
  year={2021},
  publisher={SIAM}
}

\newpage
\appendix

\section{Discussion: Relation to Anisotropic Smoothness/Convexity} 
\label{sec:anisotropic}

Laude et al. (2025) proposed the anisotropic descent (smoothness) and convexity conditions \cite{laude2025anisotropic}:

\begin{definition}\label{def:anisoconvex}
    The function $f$ is said to have the anisotropic convexity property of $\lambda$ with respect to the preconditioner function $\phi$ if, for any $x,y\in\mathbb{R}^n$, 
    \begin{equation}
    D_f(x,y)\geq D_{\phi_{\lambda}}(x-y^+,y-y^+),
    \end{equation}
    where $y^+ = y-\lambda\nabla \phi^*(\nabla f(y))$, and $\phi_{\lambda}$ is the epi-scaling of $\phi$: $\phi_{\lambda}(x) = \lambda\phi(\frac{1}{\lambda}x)$.
\end{definition}

\begin{definition}\label{def:anisosmooth}
    The function $f$ is said to have the anisotropic descent (smoothness) property of $\lambda$ with respect to the preconditioner function $\phi$ if, for any $x,y\in\mathbb{R}^n$, 
    \begin{equation}
    D_f(x,y)\leq D_{\phi_{\lambda}}(x-y^+,y-y^+),
    \end{equation}
    where $y^+ = y-\lambda\nabla \phi^*(\nabla f(y))$, and $\phi_{\lambda}$ is the epi-scaling of $\phi$: $\phi_{\lambda}(x) = \lambda\phi(\frac{1}{\lambda}x)$.
\end{definition}

We now show that \textbf{(A1$^*$)} is equivalent to having the anisotropic descent property of $L$ (Definition \ref{def:anisosmooth}) and anisotropic convexity property of $\mu$ (Definition \ref{def:anisoconvex}). Our approach to the dual mirror descent can thus be viewed as an alternative approach with greater simplicity and clarity.

\subsection{epi-scaling}
The scaling in the dual function leads to the so-called epi-scaling of $\phi$:
\[
(\phi_{\lambda})(x):=\lambda\,\phi\!\left(\frac{1}{\lambda}x\right),\qquad \lambda>0.
\]
It is called \emph{epi-scaling} because the epigraph of $\phi_\lambda$ is obtained
by scaling the epigraph of $\phi$ by the factor $\lambda$. Indeed, recall that
\[
\operatorname{epi}\phi
:=\{(x,t)\in\mathbb{R}^n\times\mathbb{R}\;:\; t\ge \phi(x)\}.
\]
For $\lambda>0$ and $\phi_\lambda(x):=\lambda\,\phi(x/\lambda)$, we have
\[
\begin{aligned}
\operatorname{epi}\phi_\lambda
&=\{(x,t)\;:\; t\ge \lambda\,\phi(x/\lambda)\} =\{(\lambda y,\lambda s)\;:\; s\ge \phi(y)\} =\lambda\,\operatorname{epi}\phi,
\end{aligned}
\]
where the last equality is understood as a uniform dilation in
$\mathbb{R}^{n+1}$.

\begin{lemma}\label{lem:epi-scaling-conjugate}
For $\lambda>0$, the convex conjugate of the epi-scaled function satisfies
\[
(\phi_{\lambda})^* = \lambda\,\phi^*.
\]
\end{lemma}
\begin{proof}
\[
\begin{aligned}
(\phi_{\lambda})^*(y)
&= \sup_{x}\Big\{\langle y,x\rangle - \lambda\,\phi(x/\lambda)\Big\}
 = \lambda \sup_{u}\Big\{\langle y,u\rangle - \phi(u)\Big\}
= \lambda\,\phi^*(y).
\end{aligned}
\]
\end{proof}

Certain scaling properties hold under epi-scaling:
\begin{itemize}
\item \noindent\textbf{Bregman divergence.}
The Bregman divergence associated with $\phi_{\lambda}$ satisfies
\[
D_{\phi_{\lambda}}(x,y)
=
\lambda\,D_\phi\!\left(\frac{x}{\lambda},\,\frac{y}{\lambda}\right),
\qquad \forall\,x,y.
\]

\item \noindent\textbf{Gradients and Hessians.}
If $\phi\in C^1$, then
\[
\nabla(\phi_{\lambda})(x)
=
\nabla\phi\!\left(\frac{x}{\lambda}\right).
\]
If $\phi\in C^2$, then
\[
\nabla^2(\phi_{\lambda})(x)
=
\frac{1}{\lambda}\,\nabla^2\phi\!\left(\frac{x}{\lambda}\right).
\]
\end{itemize}

Thus epi-scaling rescales the curvature of $\phi$, not merely its values. In
particular, it modifies smoothness and strong convexity constants while
preserving the overall geometric structure encoded by the Bregman divergence.

\subsection{Equivalence of anisotropic smoothness with dual relative smoothness}


If $\phi^*$ is $L$-relatively smooth with respect to $f^*$, define the epi-scaled
function $\psi := \phi_{1/L}$, that is, $\psi(x)=\frac{1}{L}\,\phi(Lx).$ By the conjugacy property of epi-scaling, we have
\[
D_{\psi^*}(\chi,\eta)
\le D_{f^*}(\chi,\eta),
\]
which shows that $\psi^*$ is $1$-relatively smooth with respect to $f^*$.
Therefore, by replacing $\phi$ with its epi-scaled version $\phi_{1/L}$, we may
assume without loss of generality that the relative smoothness constant satisfies
$L=1$.

\begin{proposition}
    The function $f$ has the anisotropic descent property of $\frac{1}{L}$ with respect to the preconditioner function $\phi$ if and only if $\phi^*$ is $L$-relatively smooth with respect to $f^*$, i.e., $D_{\phi^*}(\chi, \eta)\leq L D_{f^*}(\chi, \eta)$ for all $\chi, \eta\in V^*$.
\end{proposition}
\begin{proof}
Throughout, let $x,y\in\mathbb{R}^n$ be arbitrary and set
\[
\chi:=\nabla f(x),\qquad \eta:=\nabla f(y),
\qquad x=\nabla f^*(\chi),\qquad y=\nabla f^*(\eta).
\]
Assume $L=1$ (as explained via epi-scaling). We prove the equivalence between

\smallskip
\noindent\emph{(i)} $\phi^*$ is $1$-relatively smooth with respect to $f^*$, i.e.
\begin{equation}\label{eq:rel-smooth-1}
D_{\phi^*}(\eta, \chi)\le D_{f^*}(\eta, \chi),\qquad \forall\,\chi,\eta\in V^*,
\end{equation}
and

\smallskip
\noindent\emph{(ii)} $f$ has the anisotropic descent property of $1$ with respect to $\phi$, i.e.
\begin{equation}\label{eq:aniso-descent-1}
D_f(x,y)\le D_{\phi}(x-y^+,\,y-y^+),
\end{equation}
By definition,
\[
y-y^+ = \nabla \phi^*(\eta), \quad \eta = \nabla f(y), \quad \eta = \nabla \phi(y-y^+).
\]

{\bf Part: $\eqref{eq:aniso-descent-1}\ \Rightarrow\ \eqref{eq:rel-smooth-1}$.}

Fix $x,y\in\mathbb{R}^n$ and set
\[
h(z):=D_f(z,y) = f(z) - f(y) - \dual{\nabla f(y), z - y},\qquad z\in\mathbb{R}^n.
\]
Then $h$ is convex, $h(z)\ge 0$ for all $z$, and $y = \arg\min h(z)$.
Moreover, $h$ differs from $f$ by an affine term, hence they share the same Bregman divergence.

We use the anisotropic descent property to construct an upper approximation $h_{\phi}(z)$ of $h(z)$ at $z = x$. We use $h_{\phi}$ to define a gradient descent step at $x$. Then calculate the minimizer of $h_{\phi}$ and compare the difference to get \eqref{eq:rel-smooth-1}.

\medskip

\paragraph{\em Step 1: Construct an upper approximation of $h$.}
By the Bregman decomposition with base point $x$,
\begin{equation}\label{eq:h-bregman}
h(z)
=
h(x)+\langle \nabla h(x),z-x\rangle + D_h(z,x),
=
h(x)+\langle \nabla h(x),z-x\rangle + D_f(z,x),
\end{equation}
where
\[
\nabla h(x)=\nabla f(x)-\nabla f(y)=\chi - \eta.
\]

By assumption \eqref{eq:aniso-descent-1}, we have the upper bound
\[
D_f(z,x)\ \le\ D_\phi(z-x^+,x-x^+),
\]
where 
$x^+:=x-\nabla\phi^*(\chi), \nabla\phi(x-x^+)=\chi.$
Plugging this into \eqref{eq:h-bregman}, we obtain the upper model
\[
h(z)\ \le\ h_\phi(z)
:= h(x)+\langle \nabla h(x),z-x\rangle + D_\phi(z-x^+,x-x^+).
\]

\begin{figure}[htbp]
\begin{center}
\includegraphics[width=3.6in]{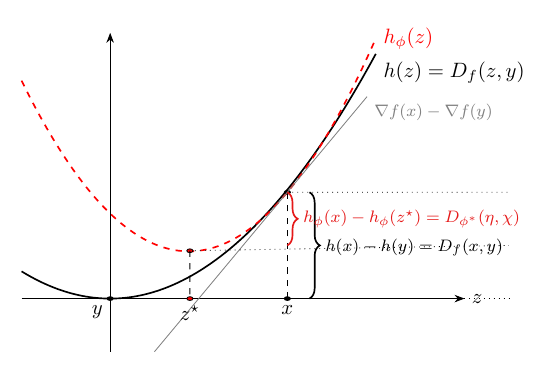}
\caption{Upper approximation of $D_f$ at $x$.}
\label{fig:upper}
\end{center}
\end{figure}

\paragraph{\em Step 2: Compute the minimizer $z^\star$ of $h_\phi$.}
We have
\[
\nabla h_\phi(z)
=
(\chi-\eta)+\nabla\phi(z-x^+)-\nabla\phi(x-x^+)
= \nabla\phi(z-x^+) - \eta.
\]
Setting $\nabla h_\phi(z^\star)=0$ gives
\[
\nabla\phi(z^\star-x^+)=\eta,
\qquad\text{so}\qquad
z^\star-x^+ = \nabla\phi^*(\eta).
\]
Using $x^+=x-\nabla\phi^*(\chi)$, we conclude
\begin{equation}\label{eq:zstar}
z^\star
- x=
\nabla\phi^*(\eta) -\nabla\phi^*(\chi).
\end{equation}

\paragraph{\em Step 3: Evaluate $h_\phi(x)-h_\phi(z^\star)$.}
First, $h_\phi(x)=h(x)$.
For $z=z^\star$, using \eqref{eq:zstar},
\[
z^\star-x
=
-\nabla\phi^*(\chi)+\nabla\phi^*(\eta),
\qquad
z^\star-x^+
=\nabla\phi^*(\eta),
\qquad
x-x^+=\nabla\phi^*(\chi),
\]
we obtain
\begin{align*}
h_\phi(z^\star)
&=
h(x)+\langle \chi-\eta,z^\star-x\rangle
+D_\phi(z^\star-x^+,x-x^+)\\
&=
h(x)+\langle \chi-\eta,\nabla\phi^*(\eta)-\nabla\phi^*(\chi)\rangle
+D_\phi\bigl(\nabla\phi^*(\eta),\nabla\phi^*(\chi)\bigr).
\end{align*}
Hence
\begin{equation}\label{eq:delta-def}
\Delta
:=
h_\phi(x)-h_\phi(z^\star)
=
-\langle \chi-\eta,\nabla\phi^*(\eta)-\nabla\phi^*(\chi)\rangle
- D_\phi\bigl(\nabla\phi^*(\eta),\nabla\phi^*(\chi)\bigr).
\end{equation}

By conjugacy of Bregman divergences,
\[
D_\phi\bigl(\nabla\phi^*(\eta),\nabla\phi^*(\chi)\bigr)
=
D_{\phi^*}(\chi,\eta).
\]
Moreover, the symmetrization identity for $\phi^*$ gives
\[
D_{\phi^*}(\eta,\chi)+D_{\phi^*}(\chi,\eta)
=
\langle \nabla\phi^*(\eta)-\nabla\phi^*(\chi),\,\eta-\chi\rangle.
\]
Therefore,
\begin{align*}
\Delta
&=
-\langle \chi-\eta,\nabla\phi^*(\eta)-\nabla\phi^*(\chi)\rangle
- D_\phi\bigl(\nabla\phi^*(\eta),\nabla\phi^*(\chi)\bigr)\\
&=
-\langle \chi-\eta,\nabla\phi^*(\eta)-\nabla\phi^*(\chi)\rangle
- D_{\phi^*}(\chi,\eta)\\
&=
D_{\phi^*}(\eta,\chi)+D_{\phi^*}(\chi,\eta)-D_{\phi^*}(\chi,\eta)
= D_{\phi^*}(\eta,\chi).
\end{align*}

\medskip
\paragraph{\em Step 4: Relate $\Delta$ to $D_f(x,y)$.}
Since $h\ge h_\phi$ pointwise and $h(y)=\min_z h(z)$, we have
\[
h(x)-h(y)\ \ge\ h(x)-\min_z h_\phi(z)\ \ge\ h(x)-h_\phi(z^\star)=\Delta.
\]
Based on Fig. \ref{fig:upper}, we obtain
\[
D_{f^*}(\eta,\chi) = D_f(x,y)
= h(x)-h(y)
\ge h(x)-h_\phi(z^\star)
= D_{\phi^*}(\eta,\chi),
\]
which is exactly \eqref{eq:rel-smooth-1}.

\medskip
{\bf Part: $\eqref{eq:rel-smooth-1}\ \Rightarrow\ \eqref{eq:aniso-descent-1}$.}

The argument is analogous to Part~1, but carried out in the dual space for the dual objective $f^*$.

Fix $x\in\mathbb{R}^n$ and let $\chi=\nabla f(x)$. Define
\[
h(v):=D_{f^*}(v,\chi)
= f^*(v)-f^*(\chi)-\langle \nabla f^*(\chi),v-\chi\rangle,
\qquad v\in\mathbb{R}^n.
\]
Then $h(v)\ge 0$ for all $v$, with $\chi=\arg\min_v h(v)$. Since $h$ differs from $f^*$ only by an affine function, they share the same Bregman divergence:
\[
D_h(u,v)=D_{f^*}(u,v), \qquad \forall\,u,v.
\]
Moreover,
\[
\nabla h(\eta)=\nabla f^*(\eta)-\nabla f^*(\chi)=y-x.
\]

\paragraph{\em Step 1: Construct a lower approximation of $h$.}
By the Bregman decomposition of $h$ at $\eta$, we have
\[
h(v)
= h(\eta)+\langle \nabla h(\eta),v-\eta\rangle + D_{f^*}(v,\eta).
\]
Using the relative smoothness assumption \eqref{eq:rel-smooth-1},
$D_{f^*}(v,\eta)\ge D_{\phi^*}(v,\eta)$, we obtain the lower model
\begin{equation}\label{eq:key-lower-bound}
h(v)\ \ge\ h_{\phi^*}(v)
:= h(\eta)+\langle \nabla h(\eta),v-\eta\rangle + D_{\phi^*}(v,\eta).
\end{equation}

\begin{figure}[htbp]
\begin{center}
\includegraphics[width=3.6in]{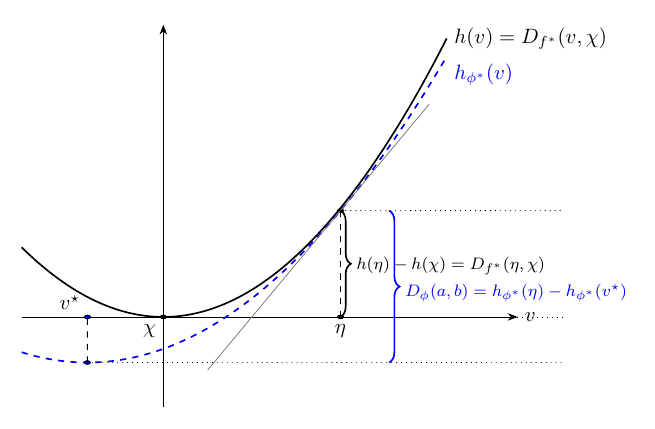}
\caption{Lower approximation of $D_{f^*}$ at $\eta$.}
\label{fig:Dflower}
\end{center}
\end{figure}

\paragraph{\em Step 2: Compute the minimizer $v^\star$ of $h_{\phi^*}$.}
The gradient of the nonconstant part of $h_{\phi^*}$ is
\[
\nabla_v\Big(\langle \nabla h(\eta),v-\eta\rangle + D_{\phi^*}(v,\eta)\Big)
= \nabla h(\eta)+\nabla\phi^*(v)-\nabla\phi^*(\eta).
\]
Setting this to zero yields
\begin{equation}\label{eq:vstar}
\nabla\phi^*(v^\star)
= \nabla\phi^*(\eta)-\nabla h(\eta)
= \nabla\phi^*(\eta)-(y-x)=x-y^+,
\end{equation}
that is,
\[
v^\star=\nabla\phi(x-y^+).
\]

\paragraph{\em Step 3: Evaluate the height of the lower model.}
Since $h_{\phi^*}(\eta)=h(\eta)$, we compute
\[
h_{\phi^*}(\eta)-h_{\phi^*}(v^\star)
= -\langle \nabla h(\eta),v^\star-\eta\rangle - D_{\phi^*}(v^\star,\eta).
\]
Let
\[
a:=x-y^+,\qquad b:=y-y^+.
\]
Then $y-x=b-a$, $v^\star=\nabla\phi(a)$, and $\eta=\nabla\phi(b)$. Hence
\[
-\langle \nabla h(\eta),v^\star-\eta\rangle
= \langle b-a,\,\nabla\phi(b)-\nabla\phi(a)\rangle.
\]
Using conjugacy,
$D_{\phi^*}(\nabla\phi(a),\nabla\phi(b))=D_\phi(b,a)$, and the symmetrized
Bregman identity, we obtain
\begin{equation}\label{eq:height-phi}
h_{\phi^*}(\eta)-h_{\phi^*}(v^\star)
= D_\phi(a,b)
= D_\phi(x-y^+,\,y-y^+).
\end{equation}

\paragraph{\em Step 4: Compare with the height of $h$.}
Since $h\ge h_{\phi^*}$ pointwise and $h(\chi)=\min_v h(v)$,
\[
h(\eta)-h(\chi)
\le h_{\phi^*}(\eta)-\min_v h_{\phi^*}(v)
= h_{\phi^*}(\eta)-h_{\phi^*}(v^\star).
\]
Therefore,
\[
D_f(x,y)
= D_{f^*}(\eta,\chi)
\le D_\phi(x-y^+,\,y-y^+),
\]
which is exactly the anisotropic descent property \eqref{eq:aniso-descent-1}.

\end{proof}

In the same spirit, we can also prove the following equivalence.
\begin{proposition}
    The function $f$ has the anisotropic convexity property of $\mu$ with respect to $\phi$ if and only if the dual function $\phi^*$ is $\mu$-relatively convex to $f^*$, i.e., $D_{\phi^*}(\chi, \eta)\geq \mu D_{f^*}(\chi, \eta)$ for all $\chi, \eta\in V^*$.
\end{proposition}

\section{Convergence Analysis for Accelerated Primal-Dual Mirror Descent}
\label{sec:acc_primal_dual}
\paragraph{Convergence Guarantees}
Accelerated linear convergence is achieved by ensuring the right-hand side of Lemma \ref{lem:descentidentity} remains non-positive through proper parameter selection. The following Young-type condition is sufficient for guaranteeing accelerated linear convergence:

\begin{assumption}[Mirror acceleration compatibility]\label{ass:compatibility}
For the iterates generated by~\eqref{eq:scheme}, the mirror map $\phi$ and the
stepsize $\alpha$ satisfy
$$
\alpha\langle \nabla f(x_{k+1}), y_k-y_{k+1}\rangle
\le
\frac1L D_{\phi^*}(\nabla f(x_{k+1}),0)
+
\mu D_\phi(y_{k+1},y_k).
$$
\end{assumption}

\begin{theorem}
  \label{thm:descentinequality}
  Assume $f-\mu\phi$ and $Lf^*-\phi^*$ are convex.  Let $\{z_k\}$ be generated
by~\eqref{eq:scheme} with $\alpha\beta=1/L$.  If Assumption~\ref{ass:compatibility}
holds, then the energy decays at a linear rate:
\begin{equation}\label{eq:descentinequality}
E(z_{k+1}) \le \frac{1}{1+\alpha} E(z_k).
\end{equation}
\end{theorem}
\begin{proof}
As $f-\mu\phi$ and $\phi$ are convex, the last two terms in \eqref{eq:descent} are non-positive and can be dropped.

We use the three-point identity for $\phi^*$ to expand the cross term $\dual{\nabla f(x_{k+1}),\nabla \phi^*(\nabla f(x_{k}))}$.

Substituting $\alpha\beta=1/L$. Since $Lf^*-\phi^*$ is convex, dual relative smoothness implies
\[
D_{\phi^*}(\nabla f(x_{k+1}),\nabla f(x_k))
\le LD_f(x_k,x_{k+1}).
\]
The cross term $\alpha\dual{\nabla f(x_{k+1}),y_k-y_{k+1}}$ is bounded by \eqref{ass:compatibility},
%
which cancels the remaining positive term in \eqref{eq:descent}. Therefore,
\[
(1+\alpha)E(z_{k+1}) \le E(z_k),
\]
which yields \eqref{eq:descentinequality}.
\end{proof}

\label{sec:dis_examples}

Next, we discuss several types of mirror maps where the assumptions needed for linear accelerated convergence holds. In all cases, \eqref{ass:compatibility} is satisfied with the step-size $\alpha$ being lower bounded by $\sqrt{\mu/L}$ up to some constant, thus guaranteeing accelerated linear convergence:
$$
E(z_{k+1}) \leq \left(1 + c\sqrt{\frac{\mu}{L}}\right)^{-1} E(z_k).
$$
 These examples cover a wide variety of mirror maps which are widely used in real practice.

\paragraph{Linear preconditioners} When $\phi(x)=\frac{1}{2}\|x\|_{B^{-1}}^2$ for some positive definite matrix $B$, the scheme \eqref{eq:scheme} reduces to the preconditioned accelerated gradient descent method:

\begin{equation}\label{eq:scheme-linearprecond}\begin{aligned}
    x_{k+1}-x_k&{}=\alpha(y_k-x_{k+1})-\frac{1}{L} B(\nabla f(x_k)),\\
    y_{k+1}-y_k&{}=\alpha(x_{k+1}-y_{k+1})-\frac{\alpha}{\mu}\nabla f(x_{k+1})).
\end{aligned}
\end{equation}
In this case, we have the following accelerated linear convergence rate.
\begin{theorem} \label{thm:linearprecond}
Under the assumptions of Theorem \ref{thm:descentinequality}, choose $\alpha=\sqrt{\mu/L}$. Then \eqref{ass:compatibility} holds, and the sequence $\{z_k\}$ generated by \eqref{eq:scheme-linearprecond} satisfies
$$
E(z_{k+1}) \leq \left(1 + \sqrt{\frac{\mu}{L}}\right)^{-1} E(z_k).
$$
\end{theorem}
\begin{proof}
    By Cauchy-Schwarz inequality and Young's inequality, we have
    $$\begin{aligned}
        \alpha\dual{\nabla f(x_{k+1}),y_{k}-y_{k+1}}{}&\leq \frac{1}{2}\left(\frac{\alpha^2}{\mu} \|\nabla f(x_{k+1})\|_{B}^2+\mu\|y_k-y_{k+1}\|_{B^{-1}}^2\right)\\
        {}&= \frac{\alpha^2}{\mu}D_{\phi^*}(\nabla f(x_{k+1}),0)+\mu D_\phi(y_{k+1},y_k),
    \end{aligned}$$
    The conclusion then follows directly from Theorem \ref{thm:descentinequality} and the fact that $\alpha=\sqrt{\mu/L}$.
\end{proof}

\paragraph{Function $\phi$ is smooth and strongly convex} In this case, the ``curvature" of $\phi$ changes slowly, and we can take it into consideration and choose $\alpha$ accordingly to guarantee the right-hand side of \eqref{eq:descentinequality} is negative.
\begin{proposition}
    Assume $\phi$ is $\mu_\phi$-strongly convex and $L_\phi$-smooth. Choose $\alpha=\sqrt{\rho_\phi}\sqrt{\frac{\mu}{L}}$, where $\rho_\phi:=\mu_\phi/L_\phi$.Then \eqref{ass:compatibility} holds, and the sequence $\{z_k\}$ generated by \eqref{eq:scheme} satisfies
$$
E(z_{k+1}) \leq \left(1 + \sqrt{\rho_\phi}\sqrt{\frac{\mu}{L}}\right)^{-1} E(z_k).
$$
\end{proposition}
\begin{proof}
By the smoothness and strong convexity of $\phi$, we have
$$D_{\phi}(x,y)\geq \frac{\mu_\phi}{2} \|x-y\|^2,\quad D_{\phi^*}(\chi,\eta)=D_\phi(\nabla \phi^*(\eta),\nabla \phi^*(\chi))\geq \frac{1}{2L_\phi} \|\chi-\eta\|^2.$$
By Cauchy-Schwarz inequality and Young's inequality, we have
$$\begin{aligned}
    \sqrt{\frac{\mu}{\kappa_\phi L_k}}\dual{\nabla f(x_{k+1}),y_{k}-y_{k+1}}{}&\leq \sqrt{\frac{\mu_\phi \mu }{L_\phi L_k}} \|\nabla f(x_{k+1})\|_*\|y_k-y_{k+1}\|\\
    {}&\leq \frac{1}{2L_\phi L_k}\|\nabla f(x_{k+1})\|_*^2+\frac{\mu_\phi \mu}{2\kappa_\phi}\|y_k-y_{k+1}\|^2\\
    {}&\leq \frac{1}{L_k} D_{\phi^*}(\nabla f(x_{k+1}),0)+\mu D_\phi(y_{k+1},y_k).
\end{aligned}$$
\end{proof}



\paragraph{Function $\phi$ is twice H\"older continuously differentiable}
When $\phi$ is not smooth, we may not be able to obtain an explicit lower bound on $\alpha$ to guarantee \eqref{ass:compatibility}. However, we still have a local accelerated linear rate.
\begin{proposition}[Local convergence]
    Assume $\phi\in C^{2,\theta}$ for some $\theta\in(0,1]$. Let $L_\phi(x)$ and $\mu_\phi(x)$ be the maximum and minimum eigenvalue of $\nabla^2\phi$ at $x$, respectively. Then for any $\hat{\kappa}>L_{\phi}(x^\star)/\mu_{\phi}(\nabla \phi^*(0
    ))$, there exists a neighborhood $\mathcal{U}$ of $x^*$ such that if $x_{k+1},y_k,y_{k+1}\in \mathcal{U}$, then there exists some $\alpha>0$ such that the inequality \eqref{ass:compatibility} holds with $\alpha=\sqrt{\frac{\mu}{\hat{\kappa} L_k}}$, and the sequence $\{z_k\}$ generated by \eqref{eq:scheme} satisfies
$$
E(z_{k+1}) \leq \left(1 + \sqrt{\frac{\mu}{\hat{\kappa} L}}\right)^{-1} E(z_k).
$$
\end{proposition}

\begin{proof}
    By the H\"older continuity of $\nabla^2 \phi$, $$\|\nabla^2\phi(x)-\nabla^2\phi(y)\|\leq L\|x-y\|^\theta.$$
    Thus for any $\chi=\nabla\phi(x),\eta=\nabla \phi(y)\in \mathrm{domain}(\phi^*)$, we have
    $$\begin{aligned}
        \|\nabla^2 \phi^*(\chi)-\nabla^2 \phi^*(\eta)\|{}&=\|(\nabla^2 \phi(x))^{-1}-(\nabla^2 \phi^*(\eta))^{-1}\|\\
        {}&=\|(\nabla^2 \phi(x))^{-1}(\nabla^2 \phi(\eta)-\nabla^2 \phi(x))(\nabla^2 \phi(\eta))^{-1}\|\\
        {}&\leq \|(\nabla^2 \phi(x))^{-1}\| \|\nabla^2 \phi(\eta)-\nabla^2 \phi(x)\| \|(\nabla^2 \phi(\eta))^{-1}\|\\
        {}&\leq \frac{L}{\mu_\phi^2}\|x-y\|^\theta\\
        {}&\leq \frac{L}{\mu_\phi^{2+\theta}}\|\chi-\eta\|^\theta.
        \end{aligned}$$
    Thus $\nabla^2 \phi^*$ is also H\"older continuous. By the remainder form of Taylor's theorem, we have
    $$\begin{aligned}
        D_{\phi^*}(\nabla f(x_{k+1}),0){}=\frac{1}{2}\|\nabla f(x_{k+1})\|_{\nabla^2 \phi^*(\xi)}^2,~~ D_{\phi}(y_{k+1},y_k){}=\frac{1}{2}\|\nabla f(y_{k+1}-y_k)\|_{\nabla^2 \phi(w)}^2,
    \end{aligned}$$
    for some $\xi$ in the dual space between $\nabla f(x_{k+1})$ and $0$, and some $w$ in the primal space between $y_{k+1}$ and $y_k$. From H\"older continuity of $\nabla^2 \phi$ and $\nabla^2 \phi^*$, we have for any $\varepsilon>0$, there exists a neighborhood $\mathcal{U}_1$ of $x^*$ such that if $y_k,y_{k+1}\in \mathcal{U}$, then $$\|\nabla^2 \phi(w)-\nabla^2 \phi(x^\star)\|\leq \varepsilon.$$ From the smoothness of $f$, there exists a neighborhood $\mathcal{U}_2$ of $x^*$ such that if $x_{k+1}\in \mathcal{U}_2$, then $$\|\nabla f(x_{k+1})-0\|_*\leq \delta,$$ for some $\delta>0$ small enough such that the corresponding $\xi$ satisfies $$\|\nabla^2 \phi^*(\xi)-\nabla^2 \phi^*(0)\|\leq \varepsilon.$$
    Let $\mathcal{U}=\mathcal{U}_1\cap \mathcal{U}_2$. Then by the arbitrary choice of $\varepsilon>0$, we have the conclusion holds for all $\hat{\kappa}>L_{\phi}(x^\star)/\mu_{\phi}(\nabla \phi^*(0))$.
\end{proof}
\section{Proof of Theorem \ref{thm:homotopy}}
\label{sec:homotopy_proof}

We use the idea from \citet{chen2025accelerated}. Denote the Lyapunov function with perturbation $\epsilon$ as
\begin{equation}\label{eq:perturbed Lyapunov}
\mathcal{E}(x, y, \epsilon) = D_f(x, x^{\star}) + \epsilon D_\phi(x^\star,y).\end{equation}
We first prove the following theorem.

\begin{theorem}
		\label{thm:convergence rate of fixed epsilon}
		Suppose $f$ is convex and $L$-dual relatively smooth to $\phi$. Furthermore, assume $\phi$ satisfies the condition \eqref{ass:compatibility} with $\alpha=\sqrt{\epsilon/L}/c_{\phi}$, where $c_{\phi}\geq 1$ is a $\phi$-dependent constant specified in Appendix \ref{sec:dis_examples}, and $\alpha_k\leq 2\sqrt{\epsilon/L}$ for all $k$. Let $(x_k, y_k)$ be the sequence generated by Algorithm \ref{alg:AAMD-0} with the initial value $(x_0, y_0)$. Assume that there exists $R >0$ such that
		\begin{equation}\label{eq:boundednorm}
			D_\phi(x^{\star}, x_k) \leq \frac{R^2}{2}, \quad \forall k \geq 0.
		\end{equation}
		Then the following decay property holds:
		\begin{equation}\label{eq: perturbed AOR-VOS Ealpha decay}
			\begin{aligned}
				\mathcal{E}(x_{k+1}, y_{k+1},\epsilon) \leq  \left(1 + \frac{1}{c_{\phi}} \sqrt{\epsilon/L}\right)^{-(k+1)}\mathcal{E}(x_{0}, y_{0},\epsilon) + 
				\epsilon R^2, \quad k \geq 0.
			\end{aligned}
		\end{equation}
	\end{theorem}
	\begin{proof}
	Apply the descent identity \eqref{eq:descent} with $\mu=\epsilon$, and use the notations of $b_k$, $i=1,2,3$, we have
    $$
    (1+\alpha_k)\mathcal E(x_{k+1}, y_{k+1},\epsilon) - \mathcal E(x_k, y_k,\epsilon) \leq \alpha_k\epsilon D_\phi(x^\star,x_{k+1}) + b_k.
    $$ With the assumption $D_\phi(x^{\star}, x_k) \leq \frac{R^2}{2}$, the following inequality holds:
		\begin{equation}
			\begin{aligned}
				\mathcal E(x_{k+1}, y_{k+1},\epsilon) &\leqslant \frac{1}{1+\alpha_k}\mathcal E(x_{k}, y_{k},\epsilon)
				+ \frac{\alpha_k \epsilon}{1+\alpha_k}\frac{R^2}{2} + \frac{1}{1+\alpha_k}b_k.
			\end{aligned}
		\end{equation}
    By the assumption, choosing $\alpha = \alpha^* = 2\sqrt{\epsilon/L}$ gives an upper bound to $\frac{\alpha_k \epsilon}{1+\alpha_k}\frac{R^2}{2}$ on the right hand side. Taking a telescoping sum, we get
        \begin{equation}
            \begin{aligned}
                \mathcal E(x_{k+1}, y_{k+1},\epsilon) &\leqslant \prod_{i=0}^{k}\left(\frac{1}{1+\alpha_i}\right)\mathcal E(x_{0}, y_{0},\epsilon) + 
                \frac{\alpha^* \epsilon }{1+\alpha^*}\frac{R^2}{2}\sum_{i=0}^{k} \frac{1}{(1+\alpha^*)^i} + p_k
            \end{aligned}
        \end{equation}
        where $p_k$ is yielded via the same argument as in Theorem \ref{thm:accumulative}. Use the fact that $\alpha_k \geq \frac{1}{c_{\phi}} \sqrt{\epsilon/L}$, thus
        \[
        \prod_{i=0}^{k}\left(\frac{1}{1+\alpha_i}\right) \leq \left(1 + \frac{1}{c_{\phi}} \sqrt{\epsilon/L}\right)^{-(k+1)}.
        \]
        The conclusion follows by noting that $p_k \leq 0$.
		\end{proof}

\begin{theorem}
    Follow the assumptions in Theorem \ref{thm:convergence rate of fixed epsilon}. Choose $(x_0, y_0)$ and $\epsilon_0$ satisfying $\mathcal E(x_0, y_0, \epsilon_0) \leq (R^2+ 1)\epsilon_{0}$.
    Then for $(x_{k}, y_{k}, \epsilon_{k})$ generated by Algorithm \ref{alg:AAMD-0}, we have
    \begin{equation}\label{eq:Ek-ek}
        \mathcal E(x_{k}, y_{k}, \epsilon_{k})\leq (R^2+1)\epsilon_{k}\quad\forall\,k\geq0,
    \end{equation}
\end{theorem}

	\begin{proof}
	We prove \eqref{eq:Ek-ek} by induction. For $k=0$, it holds by choosing $\epsilon_0 = D_f(x_0, x^{\star})$. Now suppose that $\mathcal E(x_k, y_k, \epsilon_{k})\leq (R^2+1)\epsilon_k$ and let us consider the $k+1$-th iteration. Since $0<\alpha\leq 2\sqrt{\epsilon_0/L}$ and $(1+\alpha)^{-1}\leq 1-\alpha/(1+2\sqrt{\epsilon_0/L})$, the number of inner iterations $m_{k+1}= (\sqrt{L_F}+2\sqrt{\epsilon_{0}})\ln (2(R^2+1))\epsilon_{k+1}^{-1/2}$
	is chosen so that
	\[
		\begin{aligned}
			(1+\alpha)^{-m_{k+1}} \leq{}& (1-\alpha/(1+2\sqrt{\epsilon_0/L}))^{m_{k+1}}\\
			\leq{}& \exp\left(-\alpha m_{k+1}/(1+2\sqrt{\epsilon_0/L})\right)= (2(R^2+1))^{-1}.
		\end{aligned}
	\]
	Therefore, by Theorem \eqref{thm:convergence rate of fixed epsilon},
	$$
	\begin{aligned}
		\mathcal E(x_{k+1}, y_{k+1}, \epsilon_{k+1}) &\leq \frac{1}{2(R^2+1)}\mathcal E(x_k, y_k, \epsilon_{k+1}) + \epsilon_{k+1} R^2\\
		&\leq \frac{1}{2(R^2+1)}\mathcal E(x_k, y_k, \epsilon_{k}) + \epsilon_{k+1} R^2\\
		&\leq \frac{1}{2}\epsilon_{k} + \epsilon_{k+1} R^2 = (R^2+1) \epsilon_{k+1}.
	\end{aligned}
	$$
\end{proof}

\begin{proof}[Proof of Theorem \ref{thm:homotopy}]
Since $\epsilon_k = \epsilon_0 2^{-k}$, after the $k$-th outer iteration, the overall iteration steps is
		\[
		\begin{aligned}
			M_k=\sum_{i=0}^{k}m_i ={}&  (\sqrt{L}+2\sqrt{\epsilon_{0}})\ln (2(R^2+1)) \sum_{i=1}^k \epsilon_{i}^{-1/2}\\
			={}& (\sqrt{L}+2\sqrt{\epsilon_{0}})\ln (2(R^2+1))\frac{\sqrt{2}}{\sqrt{2} - 1} \left(\epsilon_k^{-1/2}-\epsilon_0^{-1/2}\right).
		\end{aligned}
		\]
		Calculating $\epsilon_k$ from this and plugging it into \eqref{eq:Ek-ek} proves the theorem. As a result, the iteration complexity bound $O(\sqrt{L/\epsilon})$ follows easily.
\end{proof}

\section{Proof of Proposition~\ref{prop:dual_relative_smoothness}}
\label{section_polynomialnorm_proof}

\begin{proposition}[Dual relative smoothness and convexity of norm-polynomials]
\label{prop:dual_relative_smoothness}
Let $\phi,\psi:\mathbb{R}^d\to\mathbb{R}$ be convex polynomials of the norm:
\[
\phi(x) = \sum_{i=0}^m a_i \|x\|^{p_i}, \qquad
\psi(x) = \sum_{j=0}^n b_j \|x\|^{q_j},
\]
where $2 \leq p_0 < p_1 < \dots < p_m$, $2 \leq q_0 < q_1 < \dots < q_n$, and $a_i,b_j>0$. If $p_m > q_n$,
then for any $R>0$, there exists $L>0$ such that for all $y_1, y_2 \in \mathbb{R}^d$ with $\|y_1\|, \|y_2\| \geq R$,
$$D_{\phi^*}(y_1, y_2) \leq L D_{\psi^*}(y_1, y_2).$$
\end{proposition}

An intuitive interpretation of Proposition~\ref{prop:dual_relative_smoothness} is that the dual relative smoothness and convexity of norm-polynomials at the neighborhood of infinity is determined by the highest order term. The proof is done by identifying the leading terms of the convex conjugates. 

Before proceeding to the proof, we first give some remarks on the proposition.
\begin{remark}
   Analogously, if $p_0 < q_0$, then for any $R>0$, there exists $\mu>0$ such that for all $y_1,y_2$ with $\|y_1\|,\|y_2\| \le R$,
   $$D_{\phi^*}(y_1, y_2) \geq \mu D_{\psi^*}(y_1, y_2).$$
\end{remark}
\begin{remark}
   Reversing the roles of $\phi$ and $\psi$, if $p_m < q_n$, then for any $R>0$, there exists $\mu>0$ such that for all $y_1,y_2$ with $\|y_1\|,\|y_2\| \ge R$,
   $$D_{\psi^*}(y_1, y_2) \leq \frac{1}{\mu} D_{\phi^*}(y_1, y_2).$$ Combining the two arguments, if $p_m=q_n$, then $\phi^*$ and $\psi^*$ are dual relatively smooth and strongly convex with respect to each other on the domain $\{y:\|y\| \ge R\}$ for any $R>0$.
\end{remark}

The proof of the proposition is given below. It relies on the asymptotic behavior of the convex conjugate of a norm-polynomial, which is also a norm-polynomial with the highest order term determined by the highest order term of the original function.
\begin{proof}
Let $\phi(x) = c \|x\|^{p_m} + o(\|x\|^{p_m})$ with $c>0$ and $p_m>1$. Its convex conjugate is
\[
\phi^*(y) = \sup_{x \in \mathbb{R}^d} \big\{ \langle y, x \rangle - \phi(x) \big\}.
\]
For large $\|y\|$, the supremum is attained along the direction of $y$, so write $x = t \frac{y}{\|y\|}$ with $t \ge 0$, giving
\[
\phi^*(y) = \sup_{t \ge 0} \left\{ t \|y\| - c t^{p_m} + o(t^{p_m}) \right\}.
\]
Maximizing $t \|y\| - c t^{p_m}$ yields $t_* = \left( \frac{\|y\|}{c p_m} \right)^{1/(p_m-1)}$, and hence when $\|y\| \to \infty$,
\[
\phi^*(y) = \left( \frac{p_m-1}{p_m} \right) (c p_m)^{-1/(p_m-1)} \|y\|^{p_m/(p_m-1)} \big( 1 + o(1) \big).
\]
Analogously, we have \[
\psi^*(y) = \left( \frac{q_n-1}{q_n} \right) (d q_n)^{-1/(q_n-1)} \|y\|^{q_n/(q_n-1)} \big( 1 + o(1) \big).
\]
Since $p_m > q_n$, we have $\frac{p_m}{p_m-1} < \frac{q_n}{q_n-1}$. Therefore, by comparing the leading terms of $\nabla^2 \phi^*$ and $\nabla^2 \psi^*$, we get the conclusion.
\end{proof}

\section{Extension to convex composite optimization problems}
\label{sec:composite}

In this section, we consider the composite optimization problem
\begin{equation}\label{eq:composite}
    \min_{x\in\mathbb{R}^n} F(x):=f(x)+g(x),
\end{equation}
where $f$ is a smooth convex function that is $\mu$-relatively convex to $\phi$ and $g$ is a possibly non-smooth convex function with an easy-to-compute proximal operator. This formulation encompasses constrained optimization by letting g be the indicator
function of a convex set K. We will extend the AAMD method to solve problem \eqref{eq:composite} by incorporating proximal updates in the mirror descent steps. 

The accelerated primal-dual mirror descent flow for composite optimizationis defined as follows:
\begin{equation}\label{eq:flow-composite}
\begin{aligned}
    x' &\in y-x-\beta\nabla\phi^*(\nabla f(x)+\partial g(x)), \\
    \nabla \phi(y)' &\in \nabla \phi(x)-\nabla \phi(y)-\frac{1}{\mu}\left(\nabla f(x)+\partial g(x)\right).
\end{aligned}
\end{equation}
We consider the Lyapunov function
\begin{equation}\label{eq:Lya-composite}
    E(x,y)=F(x)-F(x^\star)+\mu D_\phi(x^\star,y).
\end{equation} 
We can prove the following strong Lyapunov property by an analogous argument as in the smooth case.
\begin{lemma}
     Assume $f$ is $\mu$--relatively convex to $\phi$, and $g$ is a convex function. For any $z_0=(x_0,y_0)\in\mathbb{R}^{2n}$, along the flow \eqref{eq:flow-composite}, the Lyapunov function \eqref{eq:Lya-composite} satisfies the following inequality:
    \begin{equation}
        \dual{\nabla E(z),z'}\leq -E(z)-\beta D_{\phi^*}^{\mathrm{sym}}(\nabla f(x)+q(y),0),
    \end{equation}
    where $q(y)\in \partial g(y)$, $D_{\phi^*}^{\mathrm{sym}}(\chi,0)=D_{\phi^*}(\chi,0)+D_{\phi^*}(0,\chi)$ is the symmetric Bregman divergence.
\end{lemma}
\begin{proof}
The proof is analogous to the smooth case. Let $\eta = \nabla\phi(y)$. Differentiating $E$ gives $\partial_x E = \nabla f(x)$ and 
$$\partial_\eta E = \mu (\nabla \phi^*(\nabla \phi(y)) - \nabla \phi^*(\nabla \phi(x))) = \mu(y - x^\star).$$
A direct calculation using \eqref{eq:Bregmanidentity} yields
\[
\begin{aligned}
&\dual{\nabla E(z), z'}\\
={}&\dual{\nabla f(x),y-x-\beta\nabla\phi^*(\nabla f(x))} \\
&+\mu\dual{y-x^\star,\nabla\phi(x)-\nabla\phi(y)-\tfrac1\mu\nabla f(x)} \\
={}& -\dual{\nabla f(x),x-x^\star}
-\beta\dual{\nabla f(x),\nabla\phi^*(\nabla f(x))} \\
&+\mu D_\phi(x^\star,x)-\mu D_\phi(x^\star,y)-\mu D_\phi(y,x)
\end{aligned}
\]
which yields the desired identity \eqref{eq:SL}.

When $f-\mu\phi$ is convex, the non-positive terms can be discarded, and exponential stability follows directly from the Gr{\"o}nwall's inequality..
\end{proof}

The implicit-explicit discretization of the flow \eqref{eq:flow-composite} leads to the following scheme:
\begin{equation}\label{eq:discrete-composite}
\begin{aligned}
    x_{k+1} -x_k &= \alpha(y_k - x_{k+1}) - \alpha \beta \nabla \phi^*(\nabla f(x_k) + q_{k+1}), \quad q_k \in \partial g(x_k), \\
    \nabla \phi(y_{k+1}) - \nabla \phi(y_k)&= \alpha(\nabla \phi(x_{k+1})-\nabla\phi(y_{k+1})) - \frac{\alpha}{\mu}(\nabla f(x_{k+1}) + q_{k+1}), \quad q_{k+1} \in \partial g(x_{k+1}).
\end{aligned}
\end{equation}

The resulting algorithm, which we call AAproxMD, is summarized in Alg. \ref{alg:AAproxMD}. In the above algorithm, the $x$-updates are modified to include the subgradient of $g$, which is obtained by solving the following proximal subproblem:
\begin{equation}\label{eq:proxsubproblem}
  \begin{aligned}
     z_{k+1} &{}= \frac{1}{1+\alpha_k}(x_k+\alpha_k y_k),\\
    x_{k+1} &{}= \arg\min_{x\in\mathbb{R}^n} \frac{1}{L_k(1+\alpha_k)}\phi\left(-L_k(1+\alpha_k)\left(x-z_{k+1}\right)\right)+\dual{\nabla f(x_{k}),x}+g(x).
  \end{aligned}
\end{equation}
The $y$-updates uses the subgradient $q_{k+1}$ computed in the $x$-update step. The line search procedure for $L_k$ and the spectral update remain the same as in the smooth case. In the case $\mu=0$, we employ the same homotopy strategy as in the smooth case in Appendix \ref{sec:homotopy_proof}.

\begin{example}[LASSO problem]\label{ex: Lasso}
Consider the LASSO problem
\begin{equation}
    \min_{x \in \mathbb{R}^d} F(x) := \frac{1}{2} \|Ax - b\|^2 + \lambda \|x\|_1, \quad \phi(x)=\frac{1}{2} x^\top D x,
\end{equation}
where $A \in \mathbb{R}^{n \times d}$ with $n<d$ is row full rank and $D=\mathrm{diag}(A^\top A)$. Since $D$ is diagonal and positive definite, the subproblem described in Algorithm \ref{alg:AAproxMD} admits a closed-form solution given by a generalized soft-thresholding operator. Let $L$ be the largest eigenvalues of $D^{-1/2}A^\top AD^{-1/2}$; then $f(x)=\frac{1}{2}\|Ax-b\|^2$ is 
$L$-dual relatively smooth to $\phi$. Thus, the linesearch iteration for $L_k$ terminates within the upper bound $L_k=L$ if any. Therefore, by Theorem \ref{thm:linearprecond} and the Accumulative adaptive framework, AAproxMD achieves an accelerated linear convergence rate of $O\left((1+\sqrt{\mu/L})^{-k}\right)$.

\begin{algorithm}[t]
\caption{AAproxMD}
\label{alg:AAproxMD}
\begin{spacing}{1.2} 
\begin{algorithmic}[1]
\STATE \textbf{Input:} $x_0, y_0 \in \mathbb{R}^n$ and $\mu>0$, $L_0=1$, $\alpha_0=1$, $p_{-1}=0$
\FOR{$k=0, 1, \dots$}
  \REPEAT
    \STATE $z_{k+1} = \frac{1}{1+\alpha_k}(x_k + \alpha_k y_k)$
    \STATE $x_{k+1} = \arg\min_{x\in\mathbb{R}^n} \left \{\frac{1}{L_k(1+\alpha_k)}\phi\left(-L_k(1+\alpha_k)\left(x-z_{k+1}\right)\right)+\dual{\nabla f(x_{k}),x}+g(x)\right \}$
    \STATE $q_{k+1} = \nabla \phi(L_k(1+\alpha_k)(x_{k+1}-z_{k+1})) - \nabla f(x_k)$
    \STATE $y_{k+1} = \text{prox}_{\phi}\left( \text{dual-update terms} \right)$ \COMMENT{See \eqref{eq:scheme}}
    \STATE $p_{k} = \frac{1}{1+\alpha_k}(p_{k-1} + \sum_{i=1}^3 b_k^{(i)})$
    \IF{$p_{k} > 0$}
      \STATE update $L_k$ and $\alpha_k$ via adaptive linesearch
    \ENDIF
  \UNTIL{$p_{k} \le 0$}
  \STATE $L_{k+1} = \frac{D_{\phi^*}(\nabla f(x_{k+1}) + q_{k+1}, \nabla f(x_k) + q_{k+1})}{D_F(x_k, x_{k+1})}$
  \STATE $\alpha_{k+1} = \sqrt{\mu/L_{k+1}}$
\ENDFOR
\end{algorithmic}
\end{spacing}
\end{algorithm}

The following descent property follows from the strong Lyapunov property and the comparison to the implicit Euler scheme. Its proof is similar to that of Lemma \ref{lem:descentidentity} and is thus omitted.

\begin{lemma}[Descent property]\label{lem:descentidentity-composite}
    Let $z_k=(x_k,y_k)$ be generated by \eqref{eq:scheme}. For any $z_0\in\mathbb{R}^{2n}$, the following identity holds:
    \begin{equation}\begin{aligned}
        (1+\alpha)E(z_{k+1})-E(z_k)\leq{}& -\alpha D_{f-\mu\phi}(x^\star,x_{k+1})-\alpha\mu D_\phi(y_{k+1},x_{k+1})-\alpha\beta D_{\phi^*}(0,\nabla f(x_k)+q_k)\\&{}+\alpha\beta D_\phi^*(\nabla f(x_{k+1})+q_{k+1},\nabla f(x_k)+q_k)-D_F(x_k,x_{k+1})\\
        &{}+\alpha\dual{\nabla f(x_{k+1})+q_{k+1},y_{k}-y_{k+1}}\\
        &{}-\alpha\beta D_{\phi^*}(\nabla f(x_{k+1})+q_{k+1},0)-\mu D_\phi(y_{k+1},y_k).\end{aligned}\end{equation}
\end{lemma}

We define perturbations $b_k^{(i)}$, $i=1,2,3$ and $p_k$ in the same way as in the smooth case. We establish similar convergence result for AAMD-0 in both strongly and weakly convex cases. Proof is omitted due to its similarity to the smooth case.

\begin{theorem}[strongly convex case]\label{thm:composite-stronglyconvex}
Assume $f - \mu \phi$ is convex with constant $\mu>0$, and $g$ is a possibly non-smooth convex function with an easy-to-compute proximal operator. Assume $\{z_k\}$ is produced by Algorithm \ref{alg:AAproxMD}. For the accumulative perturbation $p_k$ defined by $p_{-1}=0$ and $p_k = \frac{1}{1+\alpha_k}(p_{k-1} + b_k)$, $k\geq 0$, the Lyapunov energy \eqref{eq:Lya} satisfies:
\begin{equation}
E(z_{k+1}) \le \left( \prod_{i=0}^k \frac{1}{1+\alpha_i} \right) E(z_0).
\end{equation}
\end{theorem}
Accelerated linear convergence follows directly from Appendix \ref{sec:dis_examples}. 

\begin{algorithm}[t]
\caption{AAproxMD-0}\label{alg:AAproxMD-0}
\begin{spacing}{1.2} 
\begin{algorithmic}[1]
\STATE Initialize $x_0,y_0\in\mathbb{R}^n$, $L_0=\alpha_0=\varepsilon_0=1$, $m=m_0$, $s=0$, $k_0=0$ 
\FOR{$k=0,1,2,\dots$}
  \STATE Apply Algorithm~\ref{alg:AAproxMD} with parameter $\varepsilon_s$ for one step:
 $(x_{k+1}, y_{k+1}) = {\rm AAMDcomposite}(x_k, y_k, \varepsilon_s, 1)$
  \IF{$E_k\leq E_{k_s}/2$ \textbf{or} $k\geq k_s+m$}
    \STATE $\varepsilon_{s+1}\gets\varepsilon_s/2$, \quad $m\gets\lfloor\sqrt{2}\,m\rfloor+1$
  \ENDIF
\ENDFOR
\end{algorithmic}
\end{spacing}
\end{algorithm}

\begin{theorem}\label{thm:composite-homotopy}
Assume $f$ and $Lf^* - \phi^*$ are convex, and $g$ is a possibly non-smooth convex function with an easy-to-compute proximal operator. Let $\{x_k\}$ be generated by Algorithm \ref{alg:AAproxMD-0}. Assume $D_\phi(x^\star, x_k) \le \tfrac12 R^2$ for all $k\geq 0$. Let  $k_s$ be the total number of steps after halving $\varepsilon$ exactly $s$ times, i.e. $\varepsilon = 2^{-s}\varepsilon_0$. There exists a constant $C > 0$ so that
$$
\frac{E_{k_s}}{E_0} \leq \frac{R^2 + 1}{\left( C k_s + \varepsilon_0^{-1/2} \right)^2} = \mathcal{O}\left( \frac{1}{k_s^2} \right)
$$
So $\mathcal O(\sqrt{1/{\rm tol}})$  iteration steps can achieve $E_{k_s}/E_0\leq {\rm tol}$. 
\end{theorem}

\end{example}
\section{Boundedness of iterates}
\label{sec:boundedness}
In this section, we show that the iterates of the proposed method are bounded under mild assumptions. We first show that the iterates are bounded unconditionally in the continuous time limit, and then we show that the iterates of the discrete time method are also bounded by approximating the discrete time dynamics with the continuous time dynamics. The main idea is a new Lyapunov function
\begin{equation}\label{eq:lyapunov_appendix}
\mathcal{E}(x,y)=f(x)-f(x^\star)+\mu D_\phi(x,y).
\end{equation}
Different from \eqref{eq:Lya}, the $y$-part of the Lyapunov function is the Bregman divergence between $x$ and $y$, which is computably tractable.
\subsection{Continuous time analysis}
Consider the continuous time dynamics
\[
x'=y-x,
\]
and
\[
(\nabla \phi(y))'
=\nabla \phi(x)-\nabla \phi(y)-\frac1\mu \nabla f(x).
\]
Compared to the flow \eqref{eq:flow} in the main text, the flow here does not have the extra gradient descent term. This term is mainly used for balancing the positive terms produced by the discretization, and will be added back in the discrete time analysis. We have the following result.
\begin{theorem}
Consider the continuous time dynamics defined above. The Lyapunov function defined in \eqref{eq:lyapunov_appendix} satisfies the exact identity
\[
\mathcal{E}'
=-2\mu \,\langle \nabla\phi(x)-\nabla\phi(y),x-y\rangle.
\]
Therefore, \(\mathcal{E}'\le 0\) and the iterates are bounded.
\end{theorem}
\begin{proof}
First,
\[
\frac{d}{dt}f(x)=\langle \nabla f(x),x'\rangle
=\langle \nabla f(x),y-x\rangle .
\]

Now differentiate the Bregman divergence:
\[
D_\phi(x,y)
=\phi(x)-\phi(y)-\langle \nabla\phi(y),x-y\rangle.
\]

Using the chain rule,
\[
\begin{aligned}
\frac{d}{dt}D_\phi(x,y)
&=
\langle \nabla\phi(x),x'\rangle
-\langle \nabla\phi(y),y'\rangle \\
&\quad
-\left\langle (\nabla\phi(y))',x-y\right\rangle
-\left\langle \nabla\phi(y),x'-y'\right\rangle .
\end{aligned}
\]

The \(y'\)-terms cancel:
\[
-\langle \nabla\phi(y),y'\rangle
+\langle \nabla\phi(y),y'\rangle=0,
\]
so
\[
\frac{d}{dt}D_\phi(x,y)
=
\langle \nabla\phi(x)-\nabla\phi(y),x'\rangle
-\left\langle (\nabla\phi(y))',x-y\right\rangle.
\]

Substitute \(x'=y-x\):
\[
\begin{aligned}
\frac{d}{dt}D_\phi(x,y)
&=
\langle \nabla\phi(x)-\nabla\phi(y),y-x\rangle \\
&\quad
-\left\langle
\nabla\phi(x)-\nabla\phi(y)-\frac1\mu\nabla f(x),
x-y
\right\rangle .
\end{aligned}
\]

Since \(y-x=-(x-y)\),
\[
\langle \nabla\phi(x)-\nabla\phi(y),y-x\rangle
=
-\langle \nabla\phi(x)-\nabla\phi(y),x-y\rangle.
\]

Hence
\[
\begin{aligned}
\frac{d}{dt}D_\phi(x,y)
&=
-2\langle \nabla\phi(x)-\nabla\phi(y),x-y\rangle \\
&\quad
+\frac1\mu\langle \nabla f(x),x-y\rangle .
\end{aligned}
\]

Multiplying by \(\mu\),
\[
\mu \frac{d}{dt}D_\phi(x,y)
=
-2\mu \langle \nabla\phi(x)-\nabla\phi(y),x-y\rangle
+\langle \nabla f(x),x-y\rangle .
\]

Finally,
\[
\begin{aligned}
\mathcal{E}'
&=
\langle \nabla f(x),y-x\rangle
+\mu \frac{d}{dt}D_\phi(x,y) \\
&=
-\langle \nabla f(x),x-y\rangle
+\langle \nabla f(x),x-y\rangle \\
&\quad
-2\mu \langle \nabla\phi(x)-\nabla\phi(y),x-y\rangle.
\end{aligned}
\]

Therefore,
\[
\mathcal{E}'
=
-2\mu \,\langle \nabla\phi(x)-\nabla\phi(y),x-y\rangle .
\]

By convexity of \(\phi\),
\[
\langle \nabla\phi(x)-\nabla\phi(y),x-y\rangle \ge 0,
\]
so
\[
\mathcal{E}'\le 0.
\]

If \(\phi\) is strictly convex, then \(\mathcal{E}'=0\) iff \(x=y\).
\end{proof}

\subsection{Discrete time analysis}
Now, we analyze the discrete time method. The main idea is to approximate the discrete time dynamics with the continuous time dynamics, and then use the boundedness of the continuous time dynamics to show that the iterates of the discrete time method are also bounded. We have the following result.
\begin{theorem}
Consider the iterative scheme
\begin{align*}
x_{k+1}-x_k
&=
\alpha (y_k-x_{k+1})
-\frac{1}{L}\nabla\phi^*(\nabla f(x_k)),\\[6pt]
\nabla\phi(y_{k+1})-\nabla\phi(y_k)
&=
\alpha\big(\nabla\phi(x_{k+1})-\nabla\phi(y_{k+1})\big)
-\frac{\alpha}{\mu}\nabla f(x_{k+1}).
\end{align*}

Then the Lyapunov function \eqref{eq:lyapunov_appendix} satisfies the exact identity
\[
\mathcal{E}_{k+1}-\mathcal{E}_k
=
-\frac{1}{L}D_{\phi^*}(0,\nabla f(x_k))+O(\mu).
\]
Therefore, choosing $\mu$ sufficiently small, we have
\[
\mathcal{E}_{k+1}-\mathcal{E}_k
\le
0.
\]
\end{theorem}

\begin{proof}
Expand the Lyapunov difference at $z_{k+1}$:
\[
\begin{aligned}
\mathcal{E}_{k+1}-\mathcal{E}_k
={}&\alpha\langle \nabla \mathcal{E}(z_{k+1}),\,\mathcal{G}(z_{k+1})\rangle + \alpha\langle \nabla f(x_{k+1})+\mu(\nabla\phi(x_{k+1})-\nabla\phi(y_{k+1})),\,y_k-y_{k+1}\rangle \\&{}-\frac{1}{L}\langle \nabla f(x_{k+1})+\mu(\nabla \phi(x_{k+1})-\nabla\phi(y_{k+1})),\,\nabla\phi^*(\nabla f(x_{k}))\rangle \\
{}&-D_f(x_k,x_{k+1})-\mu(D_\phi(x_k,y_k)-D_\phi(x_{k},y_{k+1})+D_\phi(x_k,x_{k+1})\\&{}+\langle \nabla^2\phi(y_{k+1})(x_{k+1}-y_{k+1}),y_k-y_{k+1}\rangle).
\end{aligned}
\]
The last term is the Bregman divergence of $H(x,y):=D_\phi(x,y)$:
\[
\begin{aligned}
    D_H(z_k,z_{k+1})
    &=
   H(z_k)-H(z_{k+1})-\langle \nabla H(z_{k+1}),z_k-z_{k+1}\rangle \\
   &=
   D_\phi(x_k,y_k)-D_\phi(x_{k+1},y_{k+1})\\
    &\quad -\langle \nabla\phi(x_{k+1})-\nabla\phi(y_{k+1}),\,x_k-x_{k+1}\rangle
    -\langle \nabla^2\phi(y_{k+1})(x_{k+1}-y_{k+1}),y_k-y_{k+1}\rangle
\end{aligned}
\]
and 3-point identity for Bregman divergences gives
\[
\begin{aligned}
    \langle \nabla\phi(x_{k+1})-\nabla\phi(y_{k+1}),\,x_k-x_{k+1}\rangle
= D_\phi(x_k,\,y_{k+1}) - D_\phi(x_k,\,x_{k+1}) - D_\phi(x_{k+1},\,y_{k+1}).
\end{aligned}
\]

We bound each term in the above expression. First, we have
\[
\langle \nabla \mathcal{E}(z_{k+1}),\,\mathcal{G}(z_{k+1})\rangle
=-2\mu\langle \nabla\phi(x_{k+1})-\nabla\phi(y_{k+1}),\,x_{k+1}-y_{k+1}\rangle=2\mu D_\phi^{\text{sym}}(x_{k+1},y_{k+1}).
\]
Next, we assume the following inequality holds:
\[
\langle \nabla f(x_{k+1}),\,y_k-y_{k+1}\rangle
\le
\frac{1}{L} D_{\phi^*}(\nabla f(x_{k+1}),0)+\mu D_\phi(y_{k+1},y_k).
\]
Then, by the 3-point identity for Bregman divergences
\[
\begin{aligned}
&\langle \nabla f(x_{k+1}),\, \nabla\phi^*(\nabla f(x_{k}))
  \rangle \\
&\quad= D_{\phi^*}\!\left(0,\,\nabla f(x_{k})\right)
  + D_{\phi^*}\!\left(0,\,\nabla f(x_{k+1})\right)
  - D_{\phi^*}\!\left(\nabla f(x_{k+1}),\,\nabla f(x_{k})\right).
\end{aligned}
\]
By dual relative smoothness, we have
\[
D_{\phi^*}\!\left(\nabla f(x_{k+1}),\,\nabla f(x_{k})\right)
\le
L D_\phi(x_{k},x_{k+1}).
\]
Combining the above estimates, we have
\[
\begin{aligned}
\mathcal{E}_{k+1}-\mathcal{E}_k
&\le
-2\alpha\mu D_\phi^{\text{sym}}(x_{k+1},y_{k+1})\\
&+\mu\langle \nabla\phi(x_{k+1})-\nabla\phi(y_{k+1}),\,y_{k}-y_{k+1}-\frac{1}{L}\nabla \phi^*(\nabla f(x_{k}))\rangle\\
&+\mu(D_\phi(x_k,y_{k+1})-D_\phi(x_k,y_k)-D_\phi(x_k,x_{k+1})+D_\phi(y_{k+1},y_k))\\
&-\mu\langle \nabla^2\phi(y_{k+1})(x_{k+1}-y_{k+1}),y_k-y_{k+1}\rangle - \frac{1}{L}D_{\phi^*}(0,\nabla f(x_{k})).
\end{aligned}
\]
The second line can be simplified using the 3-point identity for Bregman divergences:
\[
\begin{aligned}
&D_\phi(x_k,y_{k+1})-D_\phi(x_k,y_k)-D_\phi(x_k,x_{k+1})+D_\phi(y_{k+1},y_k)\\
&{}= \langle \nabla\phi(y_k)-\nabla\phi(y_{k+1}),\, x_k - y_{k+1}\rangle 
- D_\phi(x_k, x_{k+1})
\end{aligned}
\]
Thus, we have
\[
\begin{aligned}
\mathcal{E}_{k+1}-\mathcal{E}_k
&\le
-2\alpha\mu D_\phi^{\text{sym}}(x_{k+1},y_{k+1})\\
&+\mu\langle \nabla\phi(x_{k+1})-\nabla\phi(y_{k+1})-\nabla^2\phi(y_{k+1})(x_{k+1}-y_{k+1}),\,y_{k}-y_{k+1}\rangle\\
&+\mu\langle \nabla\phi(y_k)-\nabla\phi(y_{k+1}),\, x_k - y_{k+1}\rangle 
- \mu D_\phi(x_k, x_{k+1})\\
& - \frac{1}{L}D_{\phi^*}(0,\nabla f(x_{k})) - \frac{\mu}{L}\langle \nabla\phi(x_{k+1})-\nabla\phi(y_{k+1}),\,\nabla \phi^*(\nabla f(x_{k}))\rangle.
\end{aligned}
\]
To approximate the above expression, we linearize differences of \(\nabla\phi\):
\[
\begin{aligned}
\nabla \phi(x_{k+1})-\nabla\phi(y_{k+1})=\nabla^2\phi(\xi_{k+1})(x_{k+1}-y_{k+1}),\\
\nabla\phi(y_k)-\nabla\phi(y_{k+1})=\nabla^2\phi(\zeta_{k+1})(y_k-y_{k+1}),
\end{aligned}
\]
and thus
\[
\begin{aligned}
&\langle \nabla\phi(x_{k+1})-\nabla\phi(y_{k+1})-\nabla^2\phi(y_{k+1})(x_{k+1}-y_{k+1}),\,y_{k}-y_{k+1}\rangle\\
={}&\langle (\nabla^2\phi(\xi_{k+1})-\nabla^2\phi(y_{k+1}))(x_{k+1}-y_{k+1}),\,y_k-y_{k+1}\rangle,\\
&\langle \nabla\phi(y_k)-\nabla\phi(y_{k+1}),\, x_k - y_{k+1}\rangle \\
={}&\langle y_k-y_{k+1},\, \nabla^2\phi(\zeta_{k+1})(x_k - y_{k+1})\rangle.
\end{aligned}
\]
Therefore, we have
\[
\begin{aligned}
\mathcal{E}_{k+1}-\mathcal{E}_k
&\le
-2\alpha\mu D_\phi^{\text{sym}}(x_{k+1},y_{k+1})\\
&+\mu\langle \nabla\phi(x_{k+1})-\nabla\phi(y_{k+1})-\nabla^2\phi(y_{k+1})(x_{k+1}-y_{k+1}),\,y_{k}-y_{k+1}\rangle\\
&+\mu\langle \nabla^2\phi(\zeta_{k+1})(y_k-y_{k+1}),\, x_k - y_{k+1}\rangle 
- \mu D_\phi(x_k, x_{k+1})\\
& - \frac{1}{L}D_{\phi^*}(0,\nabla f(x_{k})) - \frac{\mu}{L}\langle \nabla^2\phi(\xi_{k+1})(x_{k+1}-y_{k+1}),\,\nabla \phi^*(\nabla f(x_{k}))\rangle.
\end{aligned}
\]
\end{proof}

\section{Complete Proof of Lemma~\ref{lem:descentidentity}}
\begin{lemma}[Restatement of Lemma~\ref{lem:descentidentity}]
The iterates of~\eqref{eq:scheme} satisfy
\begin{equation}
\begin{split}
(1+\alpha)E(z_{k+1})-E(z_k)
={}&-D_f(x_k,x_{k+1})-\mu D_\phi(y_{k+1},y_k)\\
&+\alpha\dual{\nabla f(x_{k+1}),y_k-y_{k+1}}
-\alpha\beta\dual{\nabla f(x_{k+1}),\nabla\phi^*(\nabla f(x_k))}\\
&-\alpha D_{f-\mu\phi}(x^\star,x_{k+1})
-\alpha\mu D_\phi(y_{k+1},x_{k+1}).
\end{split}
\end{equation}
\end{lemma}
%
\begin{proof}
Expand the difference of Lyapunov functions at $z_{k+1}$:
$$
E(z_{k+1})-E(z_k) = \dual{\nabla E(z_{k+1}),z_k-z_{k+1}} -D_E(z_k,z_{k+1}).
$$
Since $D_E(z_k,z_{k+1})=D_f(x_k,x_{k+1})+\mu D_\phi(y_{k+1},y_k)$, we have the first line on the right-hand side of~\eqref{eq:descent}.
Let $G$ denote the vector field in~\eqref{eq:flow}. The scheme~\eqref{eq:scheme} can be written as a correction of the implicit Euler step:
$$z_{k+1}-z_k=\alpha G(z_{k+1})+\binom{\alpha(y_k-y_{k+1})-\alpha\beta\bigl(\nabla\phi^*(\nabla f(x_k))-\nabla\phi^*(\nabla f(x_{k+1}))\bigr)}{0}.$$
By Lemma \ref{lm:SL}, we have
$$
\begin{aligned}
\dual{\nabla E(z_{k+1}), \alpha G(z_{k+1})} ={}& \alpha E'(z_{k+1})\\={}&-\alpha E(z_{k+1}) - \alpha\beta\dual{\nabla f(x_{k+1}),\nabla\phi^*(\nabla f(x_{k+1}))}\\
&{} - \alpha D_{f-\mu\phi}(x^\star,x_{k+1}) - \alpha\mu D_\phi(y_{k+1},x_{k+1}).
\end{aligned}
$$
Finally, since $$\nabla E(z_{k+1})=\binom{\nabla f(x_{k+1})}{\mu\nabla\phi(y_{k+1})},$$ the remaining terms in the second and third lines in \eqref{eq:descent} are recovered by its inner product with the discrepancy terms. 
%
%
\end{proof}

\end{document}